\documentclass[11pt]{amsart}
\usepackage{txfonts}
\usepackage{amscd,amssymb,amsmath}
\usepackage[foot]{amsaddr}
\usepackage[colorlinks,plainpages,urlcolor=blue]{hyperref}
\usepackage{array}
\usepackage{enumitem}
\usepackage{tikz}
\usepackage{braids}
\usepackage{verbatim}
\usetikzlibrary{cd}
\usepackage{longtable}

\usepackage[margin=1.07in]{geometry}

\newtheorem{theorem}{Theorem}[section]
\newtheorem{lemma}[theorem]{Lemma}
\newtheorem{corollary}[theorem]{Corollary}
\newtheorem{proposition}[theorem]{Proposition}
\theoremstyle{definition}
\newtheorem{definition}[theorem]{Definition}
\newtheorem{example}[theorem]{Example}
\newtheorem{remark}[theorem]{Remark}

\DeclareMathOperator{\id}{id}

\DeclareMathOperator{\Hom}{Hom}
\DeclareMathOperator{\gr}{gr}
\DeclareMathOperator{\im}{im}
\DeclareMathOperator{\coker}{coker}

\DeclareMathOperator{\ab}{ab}

\DeclareMathOperator{\Map}{Map}

\DeclareMathOperator{\cts}{cts}

\DeclareMathOperator{\Tors}{Tors}

\DeclareMathOperator{\BK}{BK}
\DeclareMathOperator{\Heis}{Heis}

\newcommand{\mc}[1]{\mathcal{#1}}

\DeclareMathSymbol{\Gamma}{\mathalpha}{operators}{0}
\newcommand{\G}{\Gamma}

\newcommand{\Z}{\mathbb{Z}}
\newcommand{\Q}{\mathbb{Q}}
\newcommand{\R}{\mathbb{R}}

\newcommand{\F}{\mathbb{F}}

\newcommand{\T}{\mathsf{T}}

\newcommand{\bz}{\mathbf{0}}

\newcommand{\bbX}{\mathbf{X}}

\newcommand{\bY}{\mathbf{Y}}

\newcommand{\mcm}{\mc{M}}

\newcommand{\wht}[1]{\widehat{#1}}

\newcommand{\la}{\langle}
\newcommand{\ra}{\rangle}

\newcommand{\GGR}[2]{{#1}/\G^{R}_{#2}({#1})}

\newcommand{\ba}{\mathbf{a}}
\newcommand{\bb}{\mathbf{b}}
\newcommand{\bc}{\mathbf{c}}

\newcommand{\ot}{\otimes}

\DeclareMathAlphabet{\pazocal}{OMS}{zplm}{m}{n}

 \newcommand{\surj}{\twoheadrightarrow}
\newcommand{\inj}{\hookrightarrow}
\newcommand\isom{\xrightarrow{ \,\smash{\raisebox{-0.6ex}{\ensuremath{\scriptstyle\simeq}}}\,}}
\newcommand{\longisom}{\xrightarrow{\hspace{4pt}\smash{\raisebox{-0.6ex}{\ensuremath{\scriptstyle\simeq}}}\hspace{4pt}}}
\newcommand{\longsurj}{\relbar\joinrel\twoheadrightarrow}

\newcommand{\bincupdga}{{\mathrm{bin.}~\!\cup_1\text{-}\mathrm{dga}}}
\newcommand{\bincupalg}{{\mathrm{bin.}~\!\cup_1\text{-}\mathrm{alg}}}
\newcommand{\Deltaset}{{\Delta\text{-}\mathrm{set}}}

\makeatletter
\newsavebox{\@brx}
\newcommand{\llangle}[1][]{\savebox{\@brx}{\(\m@th{#1\langle}\)}%
  \mathopen{\copy\@brx\kern-0.5\wd\@brx\usebox{\@brx}}}
\newcommand{\rrangle}[1][]{\savebox{\@brx}{\(\m@th{#1\rangle}\)}%
  \mathclose{\copy\@brx\kern-0.5\wd\@brx\usebox{\@brx}}}
\makeatother

\newcommand{\arxiv}[1]
{\texttt{\href{http://arxiv.org/abs/#1}{arXiv:#1}}}

\numberwithin{table}{section}
\numberwithin{equation}{section}
\def\norms#1#2{\left\| #1 \right\|_{\lower 1ex\hbox{$\scriptstyle #2 $}}}

\makeatletter
\def\namedlabel#1#2{\begingroup
    #2%
    \def\@currentlabel{#2}%
    \phantomsection\label{#1}.\endgroup
}
\makeatother
\makeatletter
\def\@tocline#1#2#3#4#5#6#7{\relax
  \ifnum #1>\c@tocdepth 
  \else
    \par \addpenalty\@secpenalty\addvspace{#2}%
    \begingroup \hyphenpenalty\@M
    \@ifempty{#4}{%
      \@tempdima\csname r@tocindent\number#1\endcsname\relax
    }{%
      \@tempdima#4\relax
    }%
    \parindent\z@ \leftskip#3\relax \advance\leftskip\@tempdima\relax
    \rightskip\@pnumwidth plus4em \parfillskip-\@pnumwidth
    #5\leavevmode\hskip-\@tempdima
      \ifcase #1
       \or\or \hskip 1em \or \hskip 2em \else \hskip 3em \fi%
      #6\nobreak\relax
    \dotfill\hbox to\@pnumwidth{\@tocpagenum{#7}}\par
    \nobreak
    \endgroup
  \fi}
\makeatother

\begin{document}

    
\title[Groups associated to $1$-minimal models]%
{Groups associated to $1$-minimal models for \\ binomial $\cup_1$-algebras}

\author[R.~D.~Porter]{Richard~D.~Porter$^1$}
\author[A.~I.~Suciu]{Alexander~I.~Suciu$^{1,2}$}
\address{$^1$Department of Mathematics,
Northeastern University,
Boston, MA 02115, USA}
\email{\href{mailto:r.porter@northeastern.edu}{r.porter@northeastern.edu}}
\email{\href{mailto:a.suciu@northeastern.edu}{a.suciu@northeastern.edu}}
\thanks{$^2$Partially supported by the project ``Singularities and Applications" - CF 132/31.07.2023 funded by the European Union - NextGenerationEU - through Romania's National Recovery and Resilience Plan.}

\begin{abstract}
We give an explicit, cochain-level algebraic model for the pronilpotent 
completion of a group with finitely generated first cohomology. 
To each binomial $\cup_1$-dga $(A,d_A)$ over $R=\Z$ or $\F_p$ ($p$ prime)---a 
differential graded algebra endowed with a Steenrod $\cup_1$-product and a 
compatible binomial operation---we associate a pronilpotent group $G(A)$ 
that depends only on the $1$-quasi-isomorphism type of~$A$, provided 
$H^0(A)=R$ and $H^1(A)$ is a finitely generated free $R$-module. 
This group arises functorially from the $1$-minimal model of~$A$, which 
is unique up to isomorphism. When $A=C^*(X;R)$ is the cochain algebra 
of a connected CW-complex~$X$ with $H^1(X;R)$ finitely generated, the 
group $G(A)$ recovers the Bousfield--Kan $R$-completion of $\pi_1(X)$ 
when $R=\F_p$, and its pro-torsion-free-nilpotent completion when $R=\Z$.

Moreover, the group $G(A)$ comes equipped with a natural inverse system 
$\{G_n(A)\}_{n\ge 1}$ whose structure maps $G_{n+1}(A)\to G_n(A)$ 
are surjective. If $A=C^*(X;R)$, then $G_n(A)$ is the quotient of 
$\pi_1(X)$ by the $(n+1)$th term of the fastest descending central 
series whose successive quotients are free $R$-modules. We give a 
purely algebraic necessary and sufficient criterion that, given an 
isomorphism $G_n(A)\cong G_n(B)$, determines whether 
$G_{n+1}(A)\cong G_{n+1}(B)$, and we illustrate the use of this 
criterion with examples distinguishing spaces with isomorphic 
cohomology rings.
\end{abstract}



\maketitle

\setcounter{tocdepth}{1}
\tableofcontents

\section{Introduction}
\label{sect:intro}

\subsection{The group associated to a binomial $\cup_1$-dga}
\label{subsec:intro-group} 

The study of nilpotent completions of fundamental groups via
algebraic topology has a long history, going back to Malcev's
embedding of torsion-free nilpotent groups in real unipotent Lie
groups~\cite{Malcev} and Quillen's rational homotopy theory~\cite{Quillen-1969}.
In the integral and mod~$p$ settings, Sullivan's minimal
models~\cite{Sullivan} and the work of Stallings~\cite{Stallings}
and Bousfield--Kan~\cite{BK} provide the classical framework.
The present paper works in the setting of $R$-binomial $\cup_1$-dgas,
which encode the Steenrod cup-$i$ product structure on cochain 
algebras~\cite{Steenrod} together with a binomial operation
compatible with the integral lift.
We build on the theory of $1$-minimal models developed in
\cite{Porter-Suciu-2021, Porter-Suciu-2023}, which gives a unified
treatment over $\Z$ and $\F_p$.

We associate to each binomial $\cup_1$-dga $(A,d_A)$ over $R=\Z$ 
or $\F_p$ ($p$ prime), with $H^0(A)=R$ and $H^1(A)$ a finitely 
generated free $R$-module, a pronilpotent group $G(A)$.
This group arises as the inverse limit
\[
G(A) = \varprojlim_n\, G(\mcm_n(A))
\]
of a tower of nilpotent groups, one for each step $\mcm_n(A)$
in the $1$-minimal model of $A$.
We show in Theorem~\ref{thm:group-well-defined} that the
isomorphism type of $G(A)$ depends only on the
$1$-quasi-isomorphism type of $A$.

When $A = C^\ast(X;R)$ is the cochain algebra of a connected
$\Delta$-complex $X$ with fundamental group $G = \pi_1(X)$,
the groups $G(\mcm_n(A))$ can be identified with successive 
nilpotent quotients of $G$ by terms of the fastest descending 
central series whose graded pieces are free $R$-modules. 
More precisely, we prove the following.

\begin{theorem}[See Theorem~\ref{thm:mnequalsquotient}]
\label{thm:A}
Let $X$ be a connected $\Delta$-complex with $H^1(X;R)$
finitely generated, and let $G=\pi_1(X)$.
For each $n\ge 1$ there is a natural isomorphism
\[
G\bigl(\mcm_n(C^\ast(X;R))\bigr) \cong \GGR{G}{n+1},
\]
compatible with the tower maps on both sides.
\end{theorem}

Here $\GGR{G}{n+1}$ denotes the $(n+1)$th quotient in the
torsion-free lower central series (when $R=\Z$) or the
$p$-lower central series (when $R=\F_p$); see
Section~\ref{sec:group-tower} for precise definitions.
Theorem~\ref{thm:A} gives an explicit, purely algebraic 
construction---starting from $C^\ast(X;R)$ alone---of the
Stallings central series filtrations of $\pi_1(X)$.

Passing to the inverse limit, Theorem~\ref{thm:A} identifies 
$G(C^\ast(X;R))$ with a classical completion of $G$.
When $R=\F_p$, the group $G(C^\ast(X;\F_p))$ is 
the Bousfield--Kan $\F_p$-completion $\widehat{G}_{\F_p}$.
When $R=\Z$, it is the pro-torsion-free-nilpotent completion of 
$G$, which fits into a short exact sequence
\[
\begin{tikzcd}[column sep=20pt]
1 \ar[r]& \widehat{T}(G) \ar[r]&
   \widehat{G}_\Z^{\BK}\ar[r]& G(C^*(X;\Z)) \ar[r]& 1,
\end{tikzcd}
\]
relating it to the Bousfield--Kan $\Z$-completion 
(see Theorem~\ref{thm:kernel-formula}).

The key feature of the $\cup_1$ framework is that it retains 
enough integral information to recover the integral Stallings filtrations.
Classical minimal models over $\Q$ (Sullivan~\cite{Sullivan}) or $\F_p$
(Mandell~\cite{Mandell}) do not directly give the integral Stallings
series; the binomial $\cup_1$ structure provides the missing bridge.
In particular, the differential in the Hirsch extension encodes the
$k$-invariant of the Postnikov tower at the integral level, allowing
the tower $\{G(\mcm_n)\}$ to recover the torsion-free descending 
central series $\{\Gamma_n^0\}$ rather than only its rationalization.

\subsection{Isomorphisms between nilpotent quotients}
\label{subsec:intro-extend}

Our second group of results addresses the following question.
Given connected spaces $X_a$ and $X_b$ with 
$\GGR{G_a}{n} \cong \GGR{G_b}{n}$ (where $G_a = \pi_1(X_a)$
and $G_b = \pi_1(X_b)$), is there an
isomorphism $\GGR{G_a}{n+1} \cong \GGR{G_b}{n+1}$?

The first ingredient is a representability result, which identifies
dga morphisms out of $\mcm(N)$ with maps of spaces into $K(N,1)$.
In the following theorem, $BG(\mcm)$ denotes the bar construction
applied to $G(\mcm)$. 
See Definitions~\ref{def:structural-morphism} and~\ref{def:classifying} 
for structural morphisms and classifying maps, respectively.

\begin{theorem}[See Theorem~\ref{thm:strict-representability}
and Lemma~\ref{lem:correspondence}]
\label{thm:C}
Let $\mcm$ be a colimit of Hirsch extensions, and let $Y$ be a
$\Delta$-set.  There is a natural bijection
\[
\Hom_{\Deltaset}(Y,\, BG(\mcm)) \cong
\Hom_{\bincupdga} (\mcm,\,C^\ast(|Y|;R)).
\]
In particular, with $N = \GGR{G}{n}$ and $\mcm = \mcm(N)$
the $1$-minimal model for $C^\ast(K(N,1);R)$, this bijection
restricts to a bijection between classifying maps
$Y\to K(N,1)$ and structural morphisms
$\mcm(N)\to C^\ast(Y;R)$.
\end{theorem}

At the level of $1$-minimal models, this question has a precise answer.
Theorem~\ref{thm:triangle} gives a necessary and sufficient condition
for an isomorphism between the $n$th steps of two binomial
$\cup_1$-dgas to extend to an isomorphism between their $(n+1)$th steps.
Through the identification of Theorem~\ref{thm:A} and the correspondence
of Theorem~\ref{thm:C}, this algebraic criterion takes the following
topological form.

\begin{theorem}[See Theorem~\ref{thm:trianglespaces}]
\label{thm:B}
Let $X_a$ and $X_b$ be connected spaces with $H^1(X_a;R)$ and
$H^1(X_b;R)$ finitely generated, let $G_a=\pi_1(X_a)$ and
$G_b=\pi_1(X_b)$, and suppose $\GGR{G_a}{n}\cong\GGR{G_b}{n}$
for some $n\ge 2$, with common quotient $N$.
Let $q_b\colon X_b\to K(N,1)$ be a classifying map.
Then $\GGR{G_a}{n+1}\cong\GGR{G_b}{n+1}$ if and only if
there exists a classifying map $q_a\colon X_a\to K(N,1)$ and
an isomorphism of graded $R$-algebras
\[
g^{\le 2}\colon
  \im\bigl(H^{\le 2}(q_a^\#)\bigr)
  \longisom
  \im\bigl(H^{\le 2}(q_b^\#)\bigr)
\]
such that $g^2\circ H^2(q_a^\#) = H^2(q_b^\#)$.
\end{theorem}

The condition on $g^{\le 2}$ is the precise algebraic avatar of the
$k$-invariant obstruction: the isomorphism $\GGR{G_a}{n+1}\cong
\GGR{G_b}{n+1}$ exists if and only if the $k$-invariants of the two
Postnikov stages can be matched by an isomorphism of cohomology rings.
When the classifying maps $q_a^\#$ and $q_b^\#$ are surjective in
$H^2$, the condition reduces to~\cite[Thm.~6.5]{Porter-Suciu-2020}, 
which in turn relates, in the case $n=3$, to the invariant used in
the work of Rybnikov \cite{Rybnikov-1,Rybnikov-2}
to distinguish the fundamental groups of hyperplane arrangements
with the same incidence structure. Theorem~\ref{thm:trianglespaces} 
applies without any surjectivity hypothesis on the classifying maps. 

A third result, proved alongside the criterion in
Section~\ref{sec:extend}, extracts an integral invariant from the same
degree-$2$ data. When the lower central series quotients of $G$ are
torsion-free through stage~$n$, the kernel and cokernel of the structural
morphism in degree~$2$ recover, respectively, the free part and the
torsion part of the next graded piece of the lower central series.

\begin{theorem}[See Theorem~\ref{thm:coker-torsion}]
\label{thm:D}
Let $X$ be a connected $\Delta$-complex with $H^1(X;\Z)$ and
$H^2(X;\Z)$ finitely generated, let $G=\pi_1(X)$, and let
$\rho_n\colon\mcm_n\to C^\ast(X;\Z)$ be a structural morphism for the
integral $1$-minimal model. Suppose $\gr_i(G)$ is torsion-free for
$1\le i\le n$. Then $\ker H^2(\rho_n)\cong\gr_{n+1}^0(G)$ is free of
rank equal to that of $\gr_{n+1}(G)$; and if moreover
$H_2(G/\Gamma_{n+1}(G);\Z)$ is torsion-free, then
\[
\Tors\bigl(\coker H^2(\rho_n)\bigr) \cong \Tors\bigl(\gr_{n+1}(G)\bigr),
\quad\text{and}\quad
\gr_{n+1}(G) \cong \gr_{n+1}^0(G)\oplus\Tors\bigl(\gr_{n+1}(G)\bigr).
\]
\end{theorem}

Thus a single degree-$2$ map detects both the rank and the torsion of
$\gr_{n+1}(G)$: its kernel gives the free part $\gr_{n+1}^0(G)$, while the
torsion of its cokernel gives $\Tors\gr_{n+1}(G)$. The latter is
invisible to the rational $1$-minimal model and to the torsion-free
Stallings series; it coincides with the invariant $\kappa_n$
of~\cite{Porter-Suciu-2023}, which depends only on the isomorphism
type of $\pi_1(X)$, and it is the mechanism behind the second
example below.

\subsection{Examples}
\label{subsec:intro-example}

The theory can be applied to distinguish spaces $X$ and $Y$ by 
showing that, for suitable $R$ and $n$, the groups 
$\pi_1(X)/\Gamma_{n}^R(\pi_1(Y))$ and 
$\pi_1(X)/\Gamma_{n}^R(\pi_1(Y))$ are not isomorphic. 
We illustrate this in Section~\ref{sec:examples} with two examples, in
each of which the spaces involved have pairwise isomorphic cohomology
rings yet are distinguished by integral nilpotent data.

The first is a pair of $2$-complexes 
$X_0$ and $X_1$ with the same cohomology ring over $\Z$; their 
mod~$2$ nilpotent quotients $\pi_1(X_i)/\Gamma_4^2(\pi_1(X_i))$ 
are distinguished by Theorem~\ref{thm:B}, via an $\F_2$-dimension count 
on the third graded piece $\gr_3^2(\pi_1(X_i))$, detected by a Massey 
product computation in $H^\ast(X_i;\F_2)$.

The second is the infinite family of 
complements $Y(k) = S^3\setminus C(k)$ of generalized 
Borromean rings ($k\ge 1$). All pairwise linking numbers 
vanish, so the integral cohomology rings $H^\ast(Y(k);\Z)$ are 
pairwise isomorphic, and the torsion-free nilpotent quotients 
$\pi_1(Y(k))/\Gamma_s^0(\pi_1(Y(k)))$ have $k$-independent graded 
ranks for $2\le s\le 4$. Nonetheless the links are distinguished by torsion 
in the \emph{ordinary} lower central series: $\gr_3(\pi_1(Y(k))) \cong 
\Z^6 \oplus (\Z/k\Z)^2$, so $\pi_1(Y(k)) \not\cong \pi_1(Y(\ell))$ for 
$|k|\ne|\ell|$. Here, by Theorem~\ref{thm:D}, the integer Milnor 
invariant $k$ enters as the torsion of the cokernel of the structural 
morphism $H^2(\rho_2)$, a feature of the \emph{integral} $1$-minimal 
model invisible to the cohomology ring, the cup product, and the 
rational $1$-minimal model.

The companion paper \cite{Porter-Suciu-GMP} introduces a generalization 
of Massey products with smaller indeterminacy than the usual 
ones. This allows for 
stronger applications of Theorem~\ref{thm:B} in geometric contexts 
such as configuration spaces, complements of links, complements of 
toric arrangements, and complements of hyperplane arrangements.
In particular, \cite{Porter-Suciu-GMP} gives examples where generalized
Massey products tell apart spaces that are not distinguished by the 
corresponding argument using the usual Massey products.

\subsection{Organization of the paper}
\label{subsec:intro-org}

Section~\ref{sec:group-tower} reviews the relevant group theory, 
including central extensions, the torsion-free and $p$-power 
lower central series, and the Postnikov tower and Stallings sequence.
In Section \ref{sec:binomial} we recall the framework of $R$-binomial 
$\cup_1$-dgas and their $1$-minimal models, following 
\cite{Porter-Suciu-2021, Porter-Suciu-2023}.
The main objects---Hirsch extensions, colimits thereof, and 
$1$-minimal models---are defined in 
\S\S\ref{subsec:Hirsch-extensions}--\ref{subsec:1min}.
Section~\ref{sec:groups-from-hirsch} recalls the construction of 
the pronilpotent group $G(\mcm)$ associated to a colimit of Hirsch 
extensions, establishes its basic properties, and defines $G(A)$ 
for a binomial $\cup_1$-dga $A$.
Section~\ref{sect:BK} identifies $G(\mcm)$ with the 
Bousfield--Kan $R$-completion of $\pi_1(X)$ and describes 
the kernel of the natural map from the classical pronilpotent 
completion to $G(C^*(X;\Z))$.
Section~\ref{sec:main} contains the proof of Theorem~\ref{thm:A}
(Theorem~\ref{thm:mnequalsquotient} in full generality).
Section~\ref{sect:functorial} proves the strict representability 
result and the classifying map correspondence underlying
Theorem~\ref{thm:C} (Theorem~\ref{thm:strict-representability} and 
Lemma~\ref{lem:correspondence}). Section~\ref{sec:extend} 
proves Theorems~\ref{thm:B} and~\ref{thm:D}
(Theorems~\ref{thm:triangle}, \ref{thm:trianglespaces},
and~\ref{thm:coker-torsion}).
Section~\ref{sec:examples} illustrates the theory with two families
of spaces whose fundamental groups have isomorphic cohomology rings
but non-isomorphic nilpotent quotients, distinguished using
Theorem~\ref{thm:triangle} and the Massey triple product structure.

This paper is the fourth in a series; the preceding works are
\cite{Porter-Suciu-2020, Porter-Suciu-2021, Porter-Suciu-2023}.
Background results from those papers that are used here---%
particularly the existence and uniqueness of $1$-minimal models
and the homotopy lifting theorem---are collected in the Appendix
for the reader's convenience.

\section{Central series filtrations and Postnikov towers}
\label{sec:group-tower}

This section collects the group-theoretic background used throughout
the paper. We review three descending central series filtrations
of a group \(G\) (\S\ref{subsec:descending}) and the corresponding
Postnikov towers over $R$ (\S\ref{subsec:postnikov-R}). The central
extension machinery of \S\ref{subsec:central-extensions}
underlies the Hirsch extension construction in the next section.

Throughout the paper, $R$ denotes either $\Z$ or a prime field $\F_p$. 
All spaces are simplicial complexes, or more generally $\Delta$-complexes, 
as recalled in Appendix~\ref{subsec:delta}.

\subsection{Central extensions and the map $i_\chi$}
\label{subsec:central-extensions}

Let $G$ be a group and $A$ a finitely generated free $R$-module. An element 
$\chi\in H^2(G;A)$ classifies a central extension~\cite[Ch.~IV]{Brown}
\[
\begin{tikzcd}[column sep=18pt]
0 \ar[r] & A \ar[r] & \bar{G} \ar[r] & G \ar[r] & 0 .
\end{tikzcd}
\]
Geometrically, $\bar G$ is the fundamental group of the pull-back of the path-space 
fibration over $K(A,2)$ along the map $K(G,1)\to K(A,2)$ representing $\chi$.

The class $\chi$ induces a linear map
\begin{equation}
\label{eq:ichi}
i_\chi \colon H^1(A;R) \longrightarrow H^2(G;R)
\end{equation}
defined as the composition
\begin{equation}
\label{eq:ichi-comp}
\begin{tikzcd}[column sep=20pt]
H^1(A;R) \ar[r, "\cong"]
    & H^2(K(A,2);R) \ar[r, "\chi^\ast"]
    & H^2(G;R)
\end{tikzcd}
\end{equation}
This map is precisely the $d_2$-differential in the Serre spectral sequence of the 
fibration obtained from $\chi$, see e.g.~\cite[\S 5.2]{McCleary}.

\begin{lemma}
\label{lem:bij-chi-i}
The correspondence $\chi\mapsto i_\chi$ is a bijection
\[
H^2(G;A) \longisom \Hom_R\bigl(H^1(A;R),H^2(G;R)\bigr).
\]
\end{lemma}

\begin{proof}
Choose an $R$-basis $\{a_1,\dots,a_r\}$ for $A$, and let
$\{a_1^\vee,\dots,a_r^\vee\}$ be the corresponding dual basis
of $H^1(A;R) \cong \Hom_R(A,R)$. By the universal coefficient
theorem, $H^2(G;A) \cong \bigoplus_{j=1}^r H^2(G;R)$.

Given $\chi \in H^2(G;A)$, write $\chi = \sum_j \chi_j \otimes a_j$
with $\chi_j \in H^2(G;R)$. The map $i_\chi \colon H^1(A;R) \to H^2(G;R)$ 
sends $a_j^\vee$ to $\chi_j$, so $i_\chi$ is precisely the $R$-linear map
determined by $i_\chi(a_j^\vee)=\chi_j$.

Conversely, given an $R$-linear map $i \colon H^1(A;R) \to H^2(G;R)$,
set $\chi_j = i(a_j^\vee)$ and $\chi = \sum_j \chi_j \otimes a_j \in H^2(G;A)$. 
Then $i_\chi = i$ by construction. The two assignments $\chi \mapsto i_\chi$ 
and $i \mapsto \chi$ are clearly mutually inverse, establishing the bijection.
\end{proof}

\begin{lemma}
\label{lem:eG1G2}
Let $R=\Z$ or $\F_p$ with $p$ a prime, let $A$ be a finitely generated
free $R$-module, let $e \colon G_1 \to G_2$ be a homomorphism of
groups, let $\chi \in H^2(G_2;A)$, let $\bar{G}_2$ and $\bar{G}_1$
be the central extensions given by $\chi$ and $H^2(e)(\chi)$ respectively.
Then the homomorphism $e$ extends to a homomorphism $\bar{e}$
such that the following diagram commutes.
\[
\begin{tikzcd}[row sep=20pt, column sep=24pt]
0 \ar[r]
		& A \ar[r] \ar[d,  "\id" '] 
		& \bar{G}_1 \ar[r] \ar[d, "\bar{e}" '] 
		& G_1 \ar[r]  \ar[d, "{e}" '] 
		& 0\\
0 \ar[r]
		& A \ar[r] 
		& \bar{G}_2 \ar[r] 
		& G_2 \ar[r]
		& 0
\end{tikzcd}
\]
\end{lemma}

\begin{proof}
Elements of $\bar{G}_i$ are pairs $(g, a) \in G_i \times A$,
with multiplication
\[
(g,a)(h,b)=(gh, a+b-c_i(g,h)),
\]
where $c_i \colon G_i \times G_i \to A$ is a cocycle representing the
extension class. For $\bar{G}_2$, the cocycle is $c_\chi$, while 
for $\bar G_1$ it is the pullback cocycle
$c_{H^2(e)(\chi)} = c_\chi \circ (e \times e)$.
Define $\bar{e} \colon \bar{G}_1 \to \bar{G}_2$ by
$\bar{e}(g, a) = (e(g), a)$.
Then
\begin{align*}
\bar{e}\bigl((g,a)(h,b)\bigr)
&= \bar{e}(gh,\; a+b-c_1(g,h)) 
= \bar{e}(gh,\; a+b-c_\chi(e(g),e(h))) \\
&= (e(g)e(h),\; a+b-c_\chi(e(g),e(h)))
= \bar{e}(g,a)\cdot\bar{e}(h,b),
\end{align*}
so $\bar{e}$ is a homomorphism.
Commutativity of the diagram is immediate from the definition.
\end{proof}

\subsection{Descending central series filtrations}
\label{subsec:descending}

We work with three related filtrations on a group $G$:

\begin{enumerate}[label=(\roman*)]
\item The {\em lower central series} (LCS):
\[
\Gamma_1(G)=G,\qquad \Gamma_{n+1}(G)=[G,\Gamma_n(G)].
\]
The quotients $\gr_n(G)=\Gamma_n(G)/\Gamma_{n+1}(G)$ are abelian groups, 
but generally have torsion.

\item For a prime $p$, the {\em $p$-descending central series} (Stallings, \cite{Stallings}):
\[
\Gamma_1^p(G)=G,\qquad
\Gamma_{n+1}^p(G)=\langle gug^{-1}u^{-1}v^p\mid g\in G,\ u,v\in\Gamma_n^p(G)\rangle.
\]
The quotients $\gr_n^p(G)=\Gamma_n^p(G)/\Gamma_{n+1}^p(G)$ are elementary abelian 
$p$-groups (i.e., $\F_p$-vector spaces).

\item The {\em torsion-free descending central series} (Stallings, \cite[\S 7]{Stallings}):
\[
\Gamma^0_1(G)=G,
\]
and $\Gamma^0_{n+1}(G)$ is the subgroup generated by
\begin{itemize}
\item all commutators $xux^{-1}u^{-1}$ with $x\in G$, $u\in\Gamma^0_n(G)$, and
\item all $w\in\Gamma^0_n(G)$ such that some nonzero integer power $w^k$ 
$(k\neq 0)$ is a product of such commutators.
\end{itemize}
Equivalently, $\Gamma^0_{n+1}(G)$ is the smallest subgroup of $G$ such that
$\Gamma^0_n(G)/\Gamma^0_{n+1}(G)$ 
is central in $G/\Gamma^0_{n+1}(G)$ and torsion-free as an abelian group.
This is the fastest descending central series whose successive quotients
\[
\gr^0_n(G)=\Gamma^0_n(G)/\Gamma^0_{n+1}(G)
\]
are torsion-free abelian groups. 
When $H_1(G;\Z)$ is finitely generated, each $\gr^0_n(G)$ is a 
finitely generated free abelian group (hence a free $\Z$-module).
\end{enumerate}
\begin{remark}
\label{rmk:grequal}
Note that $\Gamma_n(G) \subseteq \Gamma_n^p(G)$ and
$\Gamma_n(G) \subseteq \Gamma_n^{0}(G)$ for all $n$,
since each of $\Gamma_n^p(G)$ and $\Gamma_n^0(G)$ is built using
additional generators beyond commutators (the $p$th powers in
$\Gamma_n^p(G)$, the torsion roots of commutator products in
$\Gamma_n^0(G)$). The series $\Gamma_n^p(G)$ and $\Gamma_n^0(G)$
are, however, incomparable in general: for example,
$\Gamma_2^p(\Z)=p\Z$ while $\Gamma_2^0(\Z)=0$, whereas
$\Gamma_2^p(\Z/p)=0$ while $\Gamma_2^0(\Z/p)=\Z/p$.

It follows from $\Gamma_n(G) \subseteq \Gamma_n^{0}(G)$ that if
$\gr_i(G)$ is torsion-free for $1 \le i \le n$, then
$\gr_i(G) = \gr_i^0(G)$ for $1 \le i \le n$, and hence
$\Gamma_j(G) = \Gamma_j^0(G)$ for $1 \le j \le n+1$.
\end{remark}

In the rest of the paper, $\Gamma_n(G)$ denotes the ordinary LCS,
$\Gamma_n^p(G)$ the $p$-descending central series for a prime $p$,
and $\Gamma_n^0(G)$ the torsion-free descending central series.
The corresponding graded quotients are denoted $\gr_n(G)$,
$\gr_n^p(G)$, and $\gr_n^0(G)$.

The series defined above are examples of linear-central filtrations in the sense of 
Bass and Lubotzky \cite{BL}, where the successive quotients are 
central and admit faithful linear representations over $\Z$ or $\F_p$.

\subsection{The Postnikov tower over $R$}
\label{subsec:postnikov-R}

\textbf{Convention.} Throughout the paper, $\Gamma_n^R(G)$ denotes
$\Gamma_n^0(G)$ when $R=\Z$ and $\Gamma_n^p(G)$ when $R=\F_p$.
Since the groups $\gr_n^R(G)$ are finitely generated free 
$R$-modules, the following lemma is a direct consequence 
of Stallings' $5$-term exact sequence \cite{Stallings}.

\begin{lemma}
\label{lem:3-term-seq}
Let $R=\Z$ or $\F_p$, let $G$ be a group with
$H_1(G;R)$ finitely generated, and let $\{ \Gamma_n^R(G)\}_{n\ge 1}$ 
be the fastest descending central series whose successive quotients
are free $R$-modules. Then for $Y$ any path-connected
$\Delta$-complex with $\pi_1(Y) \cong G$,
there is an exact sequence
\begin{equation}
\label{eq:St-3-term}
\begin{tikzcd}[column sep=20pt]
   0 \ar[r]    
      & H^1(\gr_n^R(G);R) \ar[r, "i_n" ]
 			& H^2(\GGR{G}{n};R) \ar[r]
 			& H^2(Y;R),
\end{tikzcd}
\end{equation}
where the central extension $\GGR{G}{n+1}$ of $\GGR{G}{n}$
is classified by $i_n$ via Lemma~\ref{lem:bij-chi-i}. 
\end{lemma}

Let $X$ be a path-connected space. Set $G=\pi_1(X)$.
The projections $G \surj \GGR{G}{n}$ and
$\GGR{G}{n+1} \surj \GGR{G}{n}$ give rise to the Postnikov tower
\eqref{eq:postnikov-3}
\begin{equation}
\label{eq:postnikov-3}
\begin{tikzcd}[row sep=24pt, column sep=60pt]
&\ar[d, dotted] \\
&  K(\GGR{G}{4},1) \ar[d, "\bar{q}_3"]\\
&  K(\GGR{G}{3},1) \ar[d, "\bar{q}_2"]\\
X \ar[r, "h_2" '] \ar[ur, "h_3" '] \ar[uur, "h_4"]
&  K(\GGR{G}{2},1)
\end{tikzcd}
\end{equation}
where the maps $h_n \colon X \to K(\GGR{G}{n},1)$ correspond to
the projection of $\pi_1(X)=G$ onto $\GGR{G}{n}$, and the
maps $\bar{q}_n$ correspond to the projection $q_n$ of 
$\GGR{G}{n+1}$ onto $\GGR{G}{n}$.
The maps $\bar{q}_n$ are fibrations with fiber 
$K(\gr_n^R (G),1)$ obtained as the pullback of the pathspace 
fibration with base $K(\gr_n^R (G),2)$ via a $k$-invariant 
\[
\chi_n \colon K(\GGR{G}{n},1) \longrightarrow K(\gr_n^R(G),2)
\]
for the central extension
\[
\begin{tikzcd}[column sep=20pt]
0\ar[r]
	&\gr_{n}^R(G) \ar[r]
	& \GGR{G}{n+1}\ar[r, "q_n"]
	& \GGR{G}{n}\ar[r]
	& 0 .
\end{tikzcd}
\]


\section{Hirsch extensions and $1$-minimal models}
\label{sec:binomial}

The main objects---Hirsch extensions, their colimits, and $1$-minimal 
models---are defined in \S\S\ref{subsec:Hirsch-extensions}--\ref{subsec:1min}.
The group associated to a colimit of Hirsch extensions is constructed in the 
following section (\S\ref{sec:groups-from-hirsch}).

\subsection{Binomial $\cup_1$-dgas}
\label{subsec:binomial-cup1}
We work in the category of \emph{$R$-binomial $\cup_1$-dgas}
as defined in \cite{Porter-Suciu-2021, Porter-Suciu-2023}.
Recall that such a dga $(A, d)$ is a graded $R$-algebra
equipped with $R$-linear cup-one operations
$\cup_1 \colon A^1 \otimes_R A^1 \to A^1$
satisfying the Hirsch identity~\cite{Hirsch}
\[
(a \cup b) \cup_1 c = a \cup (b \cup_1 c) + (a \cup_1 c) \cup b,
\qquad a,b,c \in A^1.
\]
It is also equipped with binomial operations $\zeta_n \colon A^1 \to A^1$
(with $\zeta_0 = 1$ and $\zeta_1 = \id$), where $\zeta_n$ is defined for
$n \ge 0$ in the case $R=\Z$ and for $0\le n \le p-1$ in the case
$R=\F_p$, making $R \oplus A^1$ an $R$-binomial ring. 
Finally, the differential $d$ satisfies the Leibniz rule and the
$\cup_1$-$d$ formula; see \cite[Sections~3--5]{Porter-Suciu-2023}
for the precise axioms.

The \emph{free $R$-binomial $\cup_1$-dga} on a set $\bbX$,
denoted $(\T_R(\bbX), d_{\bbX})$, is characterized by the
following two universal properties \cite{Porter-Suciu-2023}.

\begin{lemma}[{\cite[Lem.~6.2]{Porter-Suciu-2023}}]
\label{lem:extend}
Let $\bbX$ be a set, and let $A$ be an $R$-binomial $\cup_1$-graded algebra. 
Then a map of sets $\bbX \to A^1$ extends uniquely to a morphism 
$\T_R(\bbX) \to A$ of $R$-binomial $\cup_1$-graded algebras.
\end{lemma}

\begin{lemma}[{\cite[Lem.~6.3]{Porter-Suciu-2023}}]
\label{lem:commutes}
Let $(\T_R(\bbX), d_{\bbX})$ and $(A, d_{A})$ be
$R$-binomial $\cup_1$-dgas, and let $f \colon \T_R(\bbX) \to A$
be a morphism of $R$-binomial $\cup_1$-graded algebras.
Then $f$ commutes with the differentials
if and only if $d_{A}f(x) = f(d_{\bbX}x)$ for all $x \in \bbX$.
\end{lemma}

Further background appears in Appendix~\ref{app:background}.

\subsection{Hirsch extensions}
\label{subsec:Hirsch-extensions}
We recall the definition and basic properties of Hirsch extensions
from \cite{Porter-Suciu-2023}.

\begin{lemma}[{\cite[Thm.~8.2]{Porter-Suciu-2023}}]
\label{lem:Hirsch-extension}
Let $(\T_R(\bbX),d_\bbX)$ be a binomial $\cup_1$-dga and let
$f\colon\bY\to Z^2(\T_R(\bbX))$ be a set map.
There exists a unique binomial $\cup_1$-differential
$d_{\bbX\cup\bY}$ on $\T_R(\bbX\cup\bY)$ extending $d_\bbX$
and satisfying $d_{\bbX\cup\bY}(y)=f(y)$ for all $y\in\bY$.
\end{lemma}

The resulting object $(\T_R(\bbX\cup\bY),d_{\bbX\cup\bY})$
is called a \emph{Hirsch extension} of $(\T_R(\bbX),d_\bbX)$.

Given $f\colon\bY\to Z^2(\T_R(\bbX))$, set $R^{\bY}$ equal
to the free $R$-module with basis $\bY$, and let
$h \colon R^{\bY} \to H^2(\T_R(\bbX))$ be the map
sending $y \in \bY$ to the cohomology class of $f(y)$.
Then $h$ is called the \textit{$h$-invariant} of the Hirsch
extension.

\begin{lemma}
\label{lem:Hirsch-iso}
Two Hirsch extensions of $(\T_R(\bbX), d_{\bbX})$ with
isomorphic $h$-invariants are isomorphic.
More precisely, let $h_i \colon R^{\bY_i} \to H^2(\T_R(\bbX))$
be $h$-invariants of Hirsch extensions given by maps
$f_i \colon \bY_i \to Z^2(\T_R(\bbX))$ for $i=1,2$,
and let $e \colon R^{\bY_1} \isom R^{\bY_2}$
be an isomorphism of $R$-modules with $h_2 \circ e = h_1$.
Then the two Hirsch extensions are isomorphic.
\end{lemma}

\begin{proof}
From the assumptions it follows that for each $y \in \bY_1$, 
the cocycle $ f_2 \circ e (y) $ is cohomologous to $f_1(y)$ in $\T_R(\bbX)$.
For each $y \in \bY_1$, let $c(y)$ be an element in $\T_R(\bbX)$ with 
$d_{\bbX}c(y) = f_1(y) - f_2 \circ e (y)$.
Then
\begin{equation}
\label{eq:bare}
d_{\bbX \cup \bY_2}(e(y) + c(y) ) = d_{\bbX \cup \bY_1}(y)
\end{equation}
Define a map of sets 
$\overline{e}\colon \bbX \cup \bY_1 \to \bbX \cup \bY_2$ by
$x \mapsto x$ for $x \in \bbX$ and $y \mapsto e(y) + c(y)$ for
$y \in \bY_1$. 
By Lemma \ref{lem:extend} the map $\overline{e}$ extends
uniquely to a morphism
$\overline{e} \colon \T_R(\bbX \cup \bY_1) \to \T_R(\bbX \cup \bY_2)$
of $R$-binomial $\cup_1$-graded algebras.
By Lemma \ref{lem:commutes}, it follows from equation \eqref{eq:bare}
that $\overline{e}$ commutes with the differentials, and hence, is
a morphism of $R$-binomial $\cup_1$-dgas.

The remaining step is to show that $\overline{e}$ is an isomorphism.
Define the weight $|u|$ of a basis element $u \in \T_R^1(\bbX \cup \bY_1)$
to be $0$ if $u \in \T_R^1(\bbX)$ and $|u|=1$ otherwise.
Define an increasing filtration $F_1^i$,
$i \ge 0$, on  $\T_R(\bbX \cup \bY_1)$ by setting 
$F_1^i\bigl( \T_R(\bbX \cup \bY_1) \bigr)$ to be the $R$-submodule 
spanned by the tensors $u_1 \ot \cdots \ot u_\ell$ with
$\sum_{j=1}^{\ell}|u_j| \le i$.
Similarly, define weights and a filtration $F_2^i$
on $ \T_R(\bbX \cup \bY_2)$,
and note that the morphism $\overline{e}$ is filtration preserving.
Since $c(y) \in \T_R(\bbX)$, it follows that
the resulting map of quotients 
$F_1^i/F_1^{i-1}\to F_2^i/F_2^{i-1}$ 
is induced by
$x \mapsto x$ for $x \in \bbX$ and $y \mapsto e(y)$ for $y \in \bY_1$.
Since $e \colon R^{\bY_1} \to R^{\bY_2}$ is an isomorphism
it follows that the induced maps $F_1^i/F_1^{i-1}\to F_2^i/F_2^{i-1}$ 
are isomorphisms. Hence, $\overline{e}$ is an isomorphism.
\end{proof}

\begin{definition}[\cite{Porter-Suciu-2023}]
\label{def:colimit-Hirsch}
A {\em sequence of Hirsch extensions}\/ is a direct system
\begin{equation*}
\begin{tikzcd}[column sep=22pt]
\T_R(\bbX^1) \ar[r, "j_1" ] 
		& \T_R(\bbX^2) \ar[r, "j_2" ] 
		& \cdots \ar[r]
		& \T_R(\bbX^n) \ar[r, "j_n" ] 
		& \cdots ,
\end{tikzcd}
\end{equation*}
where each step is a Hirsch extension and the initial differential 
$d_{\bz}$ on $\T_R(\bbX^1)$ is given by 
$d_{\bz}x=0$ for all $x \in \bbX^1$.  
The colimit 
\[
(\T_R(\bbX),d_\bbX)=\varinjlim_n \,(\T_R(\bbX^n),d_n)
\] 
is again a binomial $\cup_1$-dga.  
If $\bbX^k=\emptyset$ for $k>N$, the colimit is \emph{finite}.
\end{definition}

\subsection{$1$-minimal models}
\label{subsec:1min}
For the rest of this section, $(A,d_A)$ will be a binomial $\cup_1$-dga 
such that  $H^0(A)=R$ and $H^1(A)$ is a finitely generated free $R$-module. 

\begin{definition}[{\cite[Def.~9.1]{Porter-Suciu-2023}}]
\label{def:1-min-model}
A \emph{$1$-minimal model} for $(A,d_A)$ consists of a 
colimit of Hirsch extensions
\[
(\mcm,d)=\varinjlim_{n}\, (\mcm_n,d_n),\qquad \mcm_n=\T_R(\bbX^n),
\]
together with compatible morphisms $\rho_n\colon\mcm_n\to A$ such that
\begin{enumerate}[itemsep=2pt]
\item \label{min1}
$H^i(\rho_1)$ is an isomorphism for $i=0,1$;
\item \label{min2}
for each $n$, the kernel $\ker H^2(\rho_n)\subset H^2(\mcm_n)$ is free with 
basis consisting of the cohomology classes of the cocycles 
$\{d_{n+1}(x)\mid x\in\bbX^{n+1}\}$.
\end{enumerate}
If $((\mcm_i, d_i), \rho_i)_{1 \le i \le n}$ satisfies properties
\eqref{min1} and \eqref{min2} for $1 \le i \le n$, then
$((\mcm_i, d_i), \rho_i)_{1 \le i \le n}$ is called 
\textit{an $n$th step in the $1$-minimal model for $A$}. 
\end{definition}

The following theorem refers to $1$-quasi-isomorphisms and 
homotopic maps. A morphism $f\colon A \to B$ of $R$-binomial 
$\cup_1$-dgas is a \emph{$1$-quasi-isomorphism} if $H^i(f)$ is 
an isomorphism for $i=0,1$ and $H^2(f)$ is a monomorphism.
Homotopy of dga maps is recalled in 
Appendix~\ref{subsec:dga-homotopy} (Definition~\ref{def:homotopy}).

\begin{theorem}[{\cite[Thm.~1.3]{Porter-Suciu-2023}}]
\label{thm:existence-uniqueness-1min}
Let $(A, d_A)$ be an $R$-binomial $\cup_1$-dga with
$H^0(A)=R$ and $H^1(A)$ a finitely generated free $R$-module.
Then
\begin{enumerate}
\item
There is a $1$-minimal model $((\mcm(A), d), \rho)$ for $A$ with
$\rho \colon \mcm(A) \to A$ a $1$-quasi-isomorphism.
\item 
Given $1$-minimal models $\rho \colon  \mcm(A) \to A$
and $\rho^\prime \colon  \mcm^\prime (A) \to A$, there is an
isomorphism $f \colon \mcm(A) \to \mcm^\prime (A)$ such that
the map
$\rho$  is homotopic to $\rho^\prime \circ f$.
\end{enumerate}
\end{theorem}


\section{Groups associated to colimits of Hirsch extensions}
\label{sec:groups-from-hirsch}

Given a colimit of Hirsch extensions $\mcm = \varinjlim_n \mcm_n$,
we construct in this section a pronilpotent group $G(\mcm)$
as the inverse limit of a compatible tower of nilpotent groups $G(\mcm_n)$.
The key structural facts---that each $G(\mcm_n)$ is nilpotent,
that the transition maps are central extensions, and that the
colimit $\mcm$ is a $1$-minimal model for the cochains of
$G(\mcm)$---are recalled from \cite{Porter-Suciu-2023}
(Lemmas~\ref{lem:Hirsch-central}--\ref{lem:dual-group} and
Theorem~\ref{thm:pronilpotent-model}).
We then define the group $G(A)$ of a binomial $\cup_1$-dga $A$
via its $1$-minimal model and show in
Theorem~\ref{thm:group-well-defined} that this is independent
of the choice of model.

When $A=C^*(BG;R)$ is the cochain algebra of the
bar construction of a finitely generated nilpotent group $G$ 
(torsion-free if $R=\Z$), it is shown in \cite[Thm.~13.1]{Porter-Suciu-2023} 
that $G(A)\cong G$.  The results above extend this canonically to the pronilpotent setting 
and to $\F_p$-coefficients. In particular, the $1$-minimal model of the cochain algebra 
$C^*(X;R)$ yields an explicit, purely algebraic description of the Stallings descending 
central series filtrations of $\pi_1(X)$ (torsion-free when $R=\Z$ and $p$-power when $R=\F_p$).

\subsection{The group of a Hirsch extension}
\label{subsec:group-hirsch}
We begin with the group structure on the set $M(\bbX; R)$
of functions $\bbX \to R$.

\begin{theorem}[{\cite[Thm. 13.1]{Porter-Suciu-2023}}]
\label{thm:product}
Let $(\mcm,d)= (\T_R(\bbX),d_{\bbX})$ be a colimit of Hirsch extensions.
Then $(M(\bbX), \nu)$ is a group, where $M(\bbX) = M(\bbX;R)$ denotes
the $R$-module of functions $\bbX \to R$ and the product
$\nu \colon M(\bbX) \times M(\bbX) \to M(\bbX)$ is defined by
\begin{equation}
\label{eq:product}
\nu(\ba, \bb)(x) = \ba(x) + \bb(x) 
						- \sum_{i=1}^{s_x}p_{x,i}(\ba)\cdot q_{x,i}(\bb)
\end{equation}
where $d_{\bbX}(x) = \sum_{i=1}^{s_x}p_{x,i}\ot q_{x,i}$
with $p_{x,i}, q_{x,i} \in \T_R^1 (\bbX)$.
\end{theorem}

The group $(M(\bbX), \nu)$ is denoted $G(\mcm)$.

\begin{definition}
\label{def:psi}
Let $(\mcm, d) = (\T_R(\bbX), d_{\bbX})$ be a colimit of Hirsch
extensions with stages $\mcm_n = \T_R(\bbX^n)$.
For each $n \ge 1$, define a morphism of binomial $\cup_1$-dgas
\[
\psi_n \colon \mcm_n \longrightarrow C^*\bigl(K(G(\mcm_n), 1); R\bigr)
\]
by setting $\psi_n(x)$ equal to the cochain in $M(\bbX^n; R)$
that sends $x$ to $1 \in R$ and all other elements of $\bbX^n$ to $0$.
These maps are compatible with the Hirsch inclusions:
$\psi_{n+1} \circ j_n = \bar{\jmath}_n^\sharp \circ \psi_n$,
and they assemble in the colimit to a morphism
\[
\psi \colon \mcm \longrightarrow C^*_{\cts}\bigl(BG(\mcm); R\bigr).
\]
\end{definition}

\begin{remark}
\label{rem:BG-vs-KG1}
For a discrete group $G$, the bar construction $B(G)$
(Appendix~\ref{subsec:delta}) is a $\Delta$-set whose geometric
realization $|B(G)|$ is a $K(G,1)$.
We write $BG$ for $B(G)$ when working at the level of
$\Delta$-sets and cochain algebras, and $K(G,1)$ when working
topologically (classifying maps, Postnikov towers, $k$-invariants).
The two are identified throughout via $|B(G)| = K(G,1)$.
\end{remark}

\begin{lemma}[{\cite[Lem.~8.2 and Lem.~8.4]{Porter-Suciu-2023}}]
\label{lem:Hirsch-central}
Let $(\mcm_n, d_n)_{n \ge 1}$ be a colimit of Hirsch extensions.
Then for each $n\ge 1$:
\begin{enumerate}[itemsep=2pt]
\item 
\label{pt:one}
$G(\mcm_n)$ is a nilpotent group (of class $\le n$).
\item 
\label{pt:two}
The inclusion $j_n\colon\mcm_n\inj\mcm_{n+1}$ induces 
a surjective homomorphism
\[
\bar{\jmath}_n\colon G(\mcm_{n+1}) \longsurj G(\mcm_n)
\]
whose kernel is central in $G(\mcm_{n+1})$ and isomorphic to $R^{\bbX_{n+1}}$.
\item 
\label{pt:three}
The $k$-invariant of the central extension
\begin{equation*}
\begin{tikzcd}[column sep=20pt]
0 \ar[r] & R^{\bbX_{n+1}} \ar[r] 
			& G(\mcm_{n+1}) \ar[r, "\bar{\jmath}_n"]
			&  G(\mcm_n) \ar[r]
			& 1 
\end{tikzcd}
\end{equation*} 
is given by the map $R^{\bbX_{n+1}} \to H^2(\mcm_n;R)$ sending $x\mapsto [d_{n+1}x]$.
\item 
\label{pt:four}
The natural map
\[
\psi_n\colon \mcm_n \longrightarrow C^*(K(G(\mcm_n),1);R)
\]
is a cohomology isomorphism and satisfies 
$\psi_{n+1}\circ j_n = \bar{\jmath}_n^{\,\sharp} \circ \psi_n$, where
for $x \in \bbX^{n+1}$, the map $\psi_{n+1}(x)$ is the map from
$\bbX^{n+1}$ to $R$ that sends $x$ to $1$ and all other elements
of $\bbX^{n+1}$ to $0$.
\end{enumerate}
\end{lemma}

\begin{lemma}[{\cite[Lem.~10.1]{Porter-Suciu-2023}}]
\label{lem:dual-group}
Let $(\T_R(\bbX), d_{\bbX})$ be a colimit of Hirsch extensions
with $\bbX = \bbX_1 \cup \cdots \cup \bbX_n \cup \cdots$ and 
$\bbX^n = \bbX_1 \cup \cdots \cup \bbX_n$. Then:
\begin{enumerate}
\item the inclusions $\T_R(\bbX^n) \to \T_R(\bbX^{n+1})$ 
induce surjective 
homomorphisms $M(\bbX^{n+1})\twoheadrightarrow M(\bbX^n)$.  
Thus, $M(\bbX)$ is a pronilpotent group (nilpotent if the colimit is finite).
\item Every morphism of colimits of Hirsch extensions induces a homomorphism 
of the associated pronilpotent groups, compatible with the inverse systems.
\end{enumerate}
\end{lemma}

\begin{theorem}
[{\cite[Thm.~8.7 and Cor.~8.8]{Porter-Suciu-2023}}]
\label{thm:central-to-hirsch}
Let $(\T_R(\bbX),d)$ be a (possibly infinite) colimit of Hirsch extensions,
and assume 
$\psi\colon\T_R(\bbX)\to C^*(M(\bbX);R)$ is a cohomology isomorphism, 
where $C^*(M(\bbX);R)$ denotes the cochains
on the bar construction  of the group $M(\bbX)$ with $R$ coefficients.

Then, if $\bar M$ is a central extension of $M(\bbX)$ 
by a finitely generated free 
$R$-module $B$, then there exists a Hirsch extension 
$(\T_R(\bbX\cup\bY),\bar d)$ 
and a map $\bar\psi\colon\T_R(\bbX\cup\bY)\to C^*(\bar M;R)$ 
that is a cohomology isomorphism.
\end{theorem}

\subsection{The pronilpotent group $G(\mcm)$}
\label{subsec:pronilp}

Applying Lemma~\ref{lem:Hirsch-central}\eqref{pt:two} 
at each stage, we obtain an inverse system of nilpotent groups
\[
\begin{tikzcd}[column sep=22pt]
G(\mcm_1) 
		& G(\mcm_2) \ar[l, "\:\bar{\jmath}_1"' ] 
		& \cdots \ar[l]
		& G(\mcm_n) \ar[l] 
		& G(\mcm_{n+1}) \ar[l, "\:\bar{\jmath}_n"' ] 
		& \ar[l]\cdots .
\end{tikzcd}
\]

\begin{definition}
\label{def:pronilp-group}
The pronilpotent group
\[
G(\mcm) \coloneqq \varprojlim_n G(\mcm_n)
\]
is called the \emph{dual group} of the colimit of Hirsch extensions 
$\mcm=\varinjlim_n \mcm_n$.  When the colimit is finite (i.e., 
$\bbX^k=\emptyset$ for $k\gg 0$), the group $G(\mcm)$ is nilpotent.
\end{definition}

We regard $G(\mcm)$ as a topological group via the inverse-limit
topology coming from the discrete topology on each $G(\mcm_n)$;
see Remark~\ref{rem:BG-vs-KG1} for the relationship between
$BG(\mcm)$ and $K(G(\mcm),1)$.

\begin{theorem}[{\cite[Thm.~8.10 and Lem.~8.11]{Porter-Suciu-2023}}]
\label{thm:pronilpotent-model}
Let $(\mcm_n,d_n)_{n\ge 1}$ be a (possibly infinite) sequence of Hirsch extensions 
with $H^1(\mcm_1;R)$ finitely generated and free. 
Then the colimit $\mcm=\varinjlim_n \mcm_n$ equipped with the natural map
\[
\psi\colon \mcm \longrightarrow C^*_{\cts} \bigl( B G(\mcm);R\bigr)
\]
is a $1$-minimal model (in the augmented sense) for the algebra 
of continuous cochains (with respect to the inverse-limit topology)
of the pronilpotent group $G(\mcm)$. In particular, $\psi$ is a 
$1$-quasi-isomorphism.
\end{theorem}

When $R=\Z$ and the colimit is finite, it is 
shown in \cite[Cor.~8.9]{Porter-Suciu-2023} 
that $G(\mcm)$ is a (finitely generated) torsion-free nilpotent group. 

\begin{example}
\label{ex:3-step}
Let $R=\Z$ and let $\bbX=\{x_1,x_2,x_{12},x_{112}\}$ 
with differential
\begin{align*}
dx_1& =dx_2=0, \\
dx_{12}&=x_1\otimes x_2, \\
dx_{112}&=x_1\otimes x_{12}-\binom{x_1}{2}\otimes x_2 .
\end{align*}
Then $\mcm=(\T_R(\bbX),d)$ is a finite colimit of Hirsch extensions, with
$\bbX_1=\{x_1,x_2\}$, $\bbX_2=\{x_{12}\}$, and $\bbX_3=\{x_{112}\}$. 
The associated group $G=G(\mcm)$ is a $3$-step nilpotent group whose underlying
set is $\Z^4$, with multiplication (in coordinates
$(a_1,a_2,a_{12},a_{112})$) given by
\begin{align*}
\mu(a,b)_1 &= a_1+b_1, \\
\mu(a,b)_2 &= a_2+b_2, \\
\mu(a,b)_{12} &= a_{12}+b_{12}-a_1 b_2, \\
\mu(a,b)_{112} &= a_{112}+b_{112}-a_1 b_{12}+\binom{a_1}{2} b_2 .
\end{align*}
The lower central series ranks of $G$ are $(2,1,1)$; in particular,
$\Gamma_3(G)/\Gamma_4(G)\cong \Z$.

The group $G$ is a proper quotient of the free $3$-step nilpotent group
$F_{2,3}=F_2/\Gamma_4(F_2)$.
Indeed, while $F_{2,3}$ has two independent weight-$3$ commutators,
$[[x,y],x]$ and $[[x,y],y]$, the Hirsch extension above encodes only the former.
Equivalently, $G$ is obtained from $F_{2,3}$ by imposing the additional relation
$[[x,y],y]=1$.

The quotient $G/\Gamma_3(G)$ is the integral Heisenberg group, and hence $G$
fits into a central extension
\[
\begin{tikzcd}[column sep=20pt]
1\ar[r]& \Z \ar[r]& G \ar[r]& \Heis(\Z) \ar[r]& 1 .
\end{tikzcd}
\]
Accordingly, $G$ is the fundamental group of a compact nilmanifold $N$ which is
a principal $S^1$-bundle over the $3$-dimensional Heisenberg nilmanifold 
$\Heis(\R)/\Heis(\Z)$.
The extension class is nontrivial in one direction and trivial in the other,
corresponding (after a choice of basis) to an Euler class of type $(1,0)$.
\end{example}

Finally, for $\mcm=\T_R(\bbX)$ with $\bbX=\{x_1, \ldots , x_\ell\}$
a finite colimit of Hirsch extensions, the elements of $G(\mcm)$ are
$\ell$-tuples $(a_1, \ldots , a_\ell) \in R^\ell$ with multiplication
\begin{equation}
\label{eq:product-2}
(\ba \cdot \bb)_i = a_i + b_i
  - \sum_{j=1}^{s_{x_i}}p_{x_i,j}(\ba)\cdot q_{x_i,j}(\bb),
\end{equation}
where $dx_i = \sum_{j} p_{x_i,j}\otimes q_{x_i,j}$.

\begin{remark}
\label{rem:product-infinite}
When $\mcm = \varinjlim_n \mcm_n$ is an infinite colimit, 
with $\bbX = \bigcup_n \bbX^n$ and each $\bbX^n$ finite, 
the finite-stage formula~\eqref{eq:product-2} applies to each 
$G(\mcm_n)$, giving a nilpotent group on the free $R$-module 
$M(\bbX^n; R)$ of functions $\bbX^n \to R$.
The pronilpotent group $G(\mcm) = \varprojlim_n G(\mcm_n)$ 
consists of all compatible families $(\ba_n)_{n \ge 1}$ with 
$\ba_n \in G(\mcm_n)$ and $\bar{\jmath}_{n}(\ba_{n+1}) = \ba_n$,
equivalently, of all functions $\ba \colon \bbX \to R$ 
whose restriction to each $\bbX^n$ lies in $G(\mcm_n)$.
Multiplication in $G(\mcm)$ is computed stagewise: 
$(\ba \cdot \bb)|_{\bbX^n} = \ba|_{\bbX^n} \cdot \bb|_{\bbX^n}$ 
in $G(\mcm_n)$ for each $n$, using formula~\eqref{eq:product-2}.
\end{remark}

\subsection{The group associated to a binomial $\cup_1$-dga}
\label{subsec:group-dga}
We now apply the construction of $G(\mcm)$ to the 
$1$-minimal model of a binomial $\cup_1$-dga.
As before, let $(A,d_A)$ be a binomial $\cup_1$-dga 
over $R=\Z$ or $\F_p$ with $H^0(A)=R$ and $H^1(A)$ 
a finitely generated free $R$-module.
 
\begin{definition}
\label{def:group-of-dga}  
Let $(\mcm(A),d)=\varinjlim_n(\mcm_n,d_n)$ be a $1$-minimal model of $A$, 
and let $\rho\colon \mcm(A)\to A$ be the associated morphism. Define
\[
G(A) \coloneqq G(\mcm(A)) = \varprojlim_n G(\mcm_n),
\]
the pronilpotent (or nilpotent, if finite) dual group of the $1$-minimal model.
\end{definition}

The canonical projections fit into the commutative diagram
\begin{equation*}
\begin{tikzcd}[column sep=2.5em, row sep=3.5em]
& & G(A) \ar[dll, two heads, "\bar\rho_1"'] \ar[dl, two heads, "\bar\rho_2"] 
\ar[d, two heads, "\bar\rho_3"] 
\ar[drr, two heads, "\bar\rho_n"] & & \\
G(\mcm_1) & G(\mcm_2)  \ar[l, two heads, pos=0.4, "\bar\jmath_1"'] 
& G(\mcm_3) \ar[l, two heads, pos=0.4, "\bar\jmath_2"'] & \cdots \ar[l, two heads] 
&  G(\mcm_n) \ar[l, two heads, pos=0.45, "\bar\jmath_{n-1}"'] & \cdots \ar[l]
\end{tikzcd}
\end{equation*}
where each $\bar\rho_n\colon G(A)\twoheadrightarrow G(\mcm_n)$ is the natural projection 
and each $\bar\jmath_k\colon G(\mcm_{k+1})\twoheadrightarrow G(\mcm_k)$ is the surjection 
induced by the Hirsch inclusion $\mcm_k\hookrightarrow\mcm_{k+1}$, 
with inverse limit isomorphism $G(A)\cong \varprojlim_n G(\mcm_n)$.

\begin{theorem}
\label{thm:group-well-defined}
The isomorphism type of $G(A)$ is independent of the choice 
of $1$-minimal model of $A$.
\end{theorem}

\begin{proof}
Let $\mcm \xrightarrow{\rho} A$ and $\mcm' \xrightarrow{\rho'} A$ be two $1$-minimal models.  
By Theorem~\ref{thm:existence-uniqueness-1min} 
there exists an isomorphism 
$f\colon\mcm \to \mcm'$ such that $\rho \simeq \rho'\circ f$.  
Since the $n$th step of a $1$-minimal model is determined inductively 
by conditions \eqref{min1} and \eqref{min2} of 
Definition~\ref{def:1-min-model}, the map $f$ restricts to isomorphisms 
$f_n\colon\mcm_n \to \mcm_n'$ at each finite stage.
By Lemma~\ref{lem:dual-group}, these induce isomorphisms 
$\bar{f}_n\colon G(\mcm_n')\to G(\mcm_n)$ that are compatible 
with the inverse systems, hence an isomorphism 
$G(\mcm')\cong G(\mcm)$.
\end{proof}


\section{Relation to Bousfield--Kan $R$-completion}
\label{sect:BK}

Having constructed the pronilpotent group $G(\mcm)$ and established
its basic properties, we place it in the context of classical completions.
We show that $G(C^\ast(X;R))$ coincides with the Bousfield--Kan
$R$-completion of $\pi_1(X)$ when $R=\F_p$, and with the
pro-torsion-free-nilpotent completion when $R=\Z$;
the gap between these two cases is measured by
Theorem~\ref{thm:kernel-formula}.

The classical Bousfield--Kan $R$-completion of a group $G$ is
the pronilpotent group
\[
\widehat{G}_R = 
\begin{cases}
\varprojlim_n G/\Gamma_n^p(G) & \text{when }R=\F_p, \\[4pt]
\varprojlim_n G/\Gamma_n(G) & \text{when }R=\Z.
\end{cases}
\]
The first is the pro-$p$ completion in the Lazard--Serre sense, 
while the second is the pronilpotent completion built from 
the ordinary lower central series (whose graded pieces may 
have torsion).

As noted in \cite[Rem.~8.12]{Porter-Suciu-2023}, the group $G(\mcm)$ 
associated to the $1$-minimal model $\mcm$ of $C^*(X;R)$ is an 
$R$-free pronilpotent group: it is the pro-torsion-free-nilpotent 
completion of $G=\pi_1(X)$ when $R=\Z$,
and equals $\widehat{G}_{\F_p}$ when $R=\F_p$.
Theorem~\ref{thm:mnequalsquotient} makes this precise at the level of towers:
the Hirsch tower $\{G(\mcm_n)\}$ is canonically isomorphic to the Stallings tower
$\{\GGR{G}{n+1}\}$ whose successive quotients are free $R$-modules.
Taking inverse limits therefore yields
\[
G(C^*(X;R)) \cong
\begin{cases}
\,\widehat{\pi_1(X)}_{\F_p} & \text{when }R=\F_p, \\[4pt]
\,\varprojlim_n \pi_1(X)/\Gamma_n^0(\pi_1(X)) & \text{when }R=\Z,
\end{cases}
\]
where $\Gamma_n^0(-)$ denotes the torsion-free Stallings series.
This gives an explicit construction, purely in terms of the cochain 
algebra $C^*(X;R)$, of the fastest descending central series filtration of 
$\pi_1(X)$ whose graded pieces are free $R$-modules.

\begin{theorem}
\label{thm:kernel-formula}
Let $X$ be a connected CW-complex with $H_1(X;\Z)$ finitely generated,
and set $G=\pi_1(X)$.
Then there is a canonical short exact sequence
\[
\begin{tikzcd}[column sep=22pt]
1 \ar[r]& \widehat{T}(G) \ar[r]&
   \widehat{G}_\Z^{\BK}\ar[r]& G(C^*(X;\Z)) \ar[r]& 1,
\end{tikzcd}
\]
where $\widehat{T}(G)$ denotes the inverse limit of the torsion 
subgroups of the finite-stage LCS quotients $G/\Gamma_n(G)$.
\end{theorem}

\begin{proof}
The ordinary LCS quotients $G/\Gamma_n(G)$ contain the full torsion of
$\gr_k(G)$ for $k\le n$. The torsion-free series quotients $G/\Gamma_n^0(G)$
are obtained from $G/\Gamma_n(G)$ by killing precisely the torsion subgroups
of $\gr_k(G)$ for $k\le n$, giving a compatible family of surjections
$G/\Gamma_n(G) \surj G/\Gamma_n^0(G)$
whose kernels are the torsion subgroups of $G/\Gamma_n(G)$.
Since the transition maps $G/\Gamma_{n+1}(G)\surj G/\Gamma_n(G)$
are surjective, the Mittag-Leffler condition is satisfied and
$\varprojlim^1$ of the torsion kernels vanishes, see \cite[Prop.~3.5.7]{Weibel}.
Passing to the inverse limit, the kernel of the natural surjection
\begin{equation}
\label{eq:vp-vp}
\widehat{G}_\Z^{\BK} = \varprojlim_n G/\Gamma_n(G) \longsurj 
\varprojlim_n G/\Gamma_n^0(G) = G(C^*(X;\Z))
\end{equation}
is therefore $\varprojlim_n \Tors(G/\Gamma_n(G)) = \widehat{T}(G)$,
as claimed.
\end{proof}

The kernel $\widehat{T}(G)$ in Theorem~\ref{thm:kernel-formula}
measures the gap between the Bousfield--Kan $\Z$-completion of
$\pi_1(X)$ and the algebraic group $G(C^*(X;\Z))$. It detects only
torsion that survives into the pronilpotent limit, not torsion that
is killed at a finite stage of the torsion-free Stallings series.
In particular, $\widehat{T}(G)$ may be trivial even when some
graded pieces $\gr_n(G)$ contain torsion, as the following
example shows.

\begin{example}
\label{ex:heisenberg-euler-k}
Let
\[
G = \langle a,b,c \mid [a,c]=[b,c]=1,\,[a,b]=c^k \rangle
\]
be the integral Heisenberg group with Euler class $k\ge 2$.
Then $G$ is torsion-free (it is a lattice in the $3$-dimensional Heisenberg
Lie group), yet $\gr_1(G) = G/\Gamma_2(G) \cong \Z^2 \oplus \Z_k$,
since the relation $[a,b]=c^k$ forces $c^k=1$ in $G_{\ab}$,
making $\bar c$ a torsion element of order $k$.
However, this torsion does not persist into $\widehat{T}(G)$: since 
$G$ is nilpotent of class~$2$, we have $\Gamma_n(G)=1$ for $n\ge 3$, 
hence $G/\Gamma_n(G) = G$ is torsion-free for $n\ge 3$. The transition 
map $\Tors(G/\Gamma_3)\to\Tors(G/\Gamma_2)$ is then zero, so 
$\widehat{T}(G)=\varprojlim_n\Tors(G/\Gamma_n)=1$.
\end{example}

By contrast, the Klein bottle provides an example where the kernel
$\widehat{T}(G)$ is a nontrivial pro-$2$ group, namely the $2$-adic
integers $\widehat{\Z}_2$. Here the torsion in the LCS graded pieces
persists at every stage, so the gap between the two completions is
substantial.

\begin{example}
\label{ex:klein}
Let $X$ be the Klein bottle, with fundamental group
$G = \pi_1(X) = \langle a,b \mid aba^{-1} = b^{-1}\rangle$.  
The ordinary lower central series is
\[
\Gamma_1(G)=G,\quad \Gamma_2(G)=\langle b^2\rangle \cong \Z,\quad
\Gamma_n(G)=\langle b^{2^{n-1}}\rangle \cong \Z \quad(n\ge 2),
\]
with graded pieces $\gr_1(G) = \Z \oplus \Z/2\Z$ and
$\gr_n(G) = \Z/2$ for $n\ge 2$.
The quotients $G/\Gamma_n(G) \cong \Z/2^{n-1}\Z \rtimes \Z$ for $n\ge 2$,
so $\widehat{G}_\Z^{\BK} \cong \widehat{\Z}_2 \rtimes \Z$,
where $\widehat{\Z}_2=\varprojlim_n \Z/2^{n}\Z$ denotes the $2$-adic integers.

The torsion-free Stallings series kills the $\Z/2$ torsion in $\gr_1(G)$
at the first step and then stabilizes:
\[
\Gamma_1^0(G)=G,\qquad \Gamma_n^0(G)=\langle b\rangle \cong \Z 
\quad \text{for all } n \ge 2.
\]
Indeed, every commutator $[a, b^j] = b^{-2j}$ lies in $\langle b^2\rangle$, 
and every element of $\langle b\rangle$ has a positive integer power in 
$\langle b^2\rangle$, so $\Gamma_3^0 = \Gamma_2^0 = \langle b\rangle$,
and similarly at higher stages.
Consequently $G/\Gamma_n^0(G) = G/\langle b\rangle \cong \Z$ for all $n\ge 2$,
so $G(C^*(X;\Z)) \cong \Z$.
The canonical surjection~\eqref{eq:vp-vp} then has kernel $\widehat{\Z}_2$,
yielding the short exact sequence
\[
\begin{tikzcd}[column sep=22pt]
1 \ar[r]& \widehat{\Z}_2 \ar[r]& \widehat{\Z}_2 \rtimes \Z \ar[r]& \Z \ar[r]& 1.
\end{tikzcd}
\]
\end{example}


\section{Structural morphisms and the main theorem}
\label{sec:main}

The central result of this paper is Theorem~\ref{thm:mnequalsquotient}, which 
identifies the tower of nilpotent groups $\{G(\mcm_n)\}_{n \ge 1}$ 
arising from the $1$-minimal model of $C^*(X;R)$ with the tower 
$\{\GGR{\pi_1(X)}{n+1}\}_{n \ge 1}$ of nilpotent quotients 
of $\pi_1(X)$. 
The proof proceeds by induction on $n$, with Lemma~\ref{lem:classify} 
as the key technical input. We establish Lemma~\ref{lem:classify} in 
\S\ref{subsec:structural-morphisms}, then prove 
Theorem~\ref{thm:mnequalsquotient} in \S\ref{subsec:proof-main}.

\subsection{Structural morphisms}
\label{subsec:structural-morphisms}

We introduce the notion of a structural morphism, which generalizes
the structural map $\rho_n \colon \mcm_n \to A$ of the $1$-minimal
model to maps that need not arise from the model itself.

\begin{definition}
\label{def:structural-morphism}
Let $n \ge 1$, and let $\bigl( (\mcm_i,d_i), \rho_i \bigr)_{1 \le i \le n}$ 
be an $n$th step in the $1$-minimal model for $A$. A morphism 
$\rho_n^\prime \colon \mcm_n  \to A$
is called a \textit{structural morphism from $\mcm_n$ to $A$}
if $H^i(\rho_1^\prime)$ is an isomorphism for $i=0,1$, where,
as indicated in diagram \eqref{eq:classify},
$\rho_1^\prime = \rho_n^\prime\circ j_{1,n}$ with
$j_{1,n}$ equal to the composition of the inclusion maps
$j_{n-1} \circ \cdots \circ j_2 \circ j_1$.
\begin{equation}
\label{eq:classify}
\begin{tikzcd}[row sep=24pt, column sep=30pt]
  	& \mcm_{n}
  		\ar[dl, "\rho_n^\prime" ']      \\
A  & \mcm_1 
		\ar[l, "\rho_1^\prime" ] 
		\ar[u, "j_{1,n}" '] 
\end{tikzcd}
\end{equation}
\end{definition}

The base case of the induction uses the following 
bijection from \cite[Cor.~7.2]{Porter-Suciu-2023}.

\begin{lemma}[{\cite[Cor.~7.2]{Porter-Suciu-2023}}]
\label{lem:H1}
If $(A,d_A)$ is an $R$-binomial $\cup_1$-dga with $H^0(A)=R$, then
there is a bijection between $R$-binomial $\cup_1$-dga maps from
$(\T_R(\bbX), d_{\bz})$ to $(A, d_A)$ and maps of sets from
$\bbX$ to $Z^1(A)$.
\end{lemma} 

We can now state and prove the key lemma.

\begin{lemma}
\label{lem:classify}
Let $\bigl((\mcm_n, d_n), \rho_n\bigr)_{n \ge 1}$ be a $1$-minimal
model for $A$, and let $\rho_n' \colon \mcm_n \to A$ be a 
structural morphism. Then, 
\begin{enumerate}[label=(\roman*), itemsep=2pt]
\item
\label{cl:i}
$( (\mcm_j,d_j), \rho_j^\prime)_{1 \le j \le n}$ is an $n$th step in
a $1$-minimal model for $A$, where $\rho_j^\prime$ denotes
the restriction of $\rho_n^\prime$ to $\mcm_j$ for $1 \le j <n$.
\item\label{cl:ii}
There is an isomorphism $f_n \colon \mcm_n \isom \mcm_n$
and a homotopy $\Phi_n$ between $\rho_n'$ and $\rho_n \circ f_n$.
\item
\label{cl:iii}
$\mcm_{n+1}$ is the Hirsch extension of $\mcm_n$ with
$h$-invariant $\ker ( H^2(\rho_n^\prime) \hookrightarrow H^2(\mcm_n))$.
\end{enumerate}
\end{lemma}

\begin{proof}
We prove statements~\ref{cl:i}, \ref{cl:ii}, and \ref{cl:iii} simultaneously
by induction on $n$. 

\medskip
\noindent\textit{Base case $(n=1)$.}
From the definition of structural morphism, $H^{i}(\rho_1^\prime)$
is an isomorphism for $i=0,1$. 
Thus, $((\mcm_1, d_1), \rho_1^\prime)$ is a first
step in the $1$-minimal model for $A$ and statement \ref{cl:i}
holds for $n=1$.

Now consider statement \ref{cl:ii}.
Since $\mcm_1 = \T_R(\bbX_1)$ with $d_1 = 0$ on generators, we have
$H^1(\mcm_1) = Z^1(\mcm_1)$, the free $R$-module on $\bbX_1$.
Both $\rho_1$ and $\rho_1'$ induce isomorphisms on $H^1$, so there
is an isomorphism $e \colon H^1(\mcm_1) \to H^1(\mcm_1)$ with
$H^1(\rho_1') = e \circ H^1(\rho_1)$.
By Lemma~\ref{lem:H1}, there is a unique dga morphism
$f_1 \colon \mcm_1 \to \mcm_1$ with $f_1(x) = e(x)$ for all
$x \in \bbX_1$. Since $e$ is an isomorphism and $\mcm_1$ is free 
on $\bbX_1$, $f_1$ is an isomorphism, with $H^1(f_1) = e$.
Hence, $H^1(\rho_1') = H^1(\rho_1 \circ f_1)$, and
Lemma~\ref{lem:H1-homotopy} provides a homotopy $\Phi_1$ between
$\rho_1'$ and $\rho_1 \circ f_1$. Thus, statement \ref{cl:ii} holds for
$n=1$.

Now consider statement \ref{cl:iii}.
$(\mcm_1, d_1) = (\T_R(\bbX_1), d_{\bz})$ for some finite set
$\bbX_1$, and from the argument above we have the following
commutative diagram with exact rows
\begin{equation}
\label{eq:cd-22}
\begin{tikzcd}[column sep=30pt]
0 \ar[r] & H^1(\T_R(\bbX_{1}), d_{\bz}) \ar[r, "e_1"]
    & H^2(\mcm_1) \ar[r, "H^2(\rho_1)"]
    & H^2(A) \\
0 \ar[r] & \ker H^2(\rho_1')
    \ar[r, "g_1"]
    & H^2(\mcm_1)
        \ar[u, "\cong"', "H^2(f_1)"]
        \ar[r, "H^2(\rho_1')"]
    & H^2(A), \ar[u, "\id"]
\end{tikzcd}
\end{equation}
where by the definition of $1$-minimal model it follows that
 $e_1$ is the $h$-invariant of the Hirsch extension
$\mcm_1 \hookrightarrow \mcm_{2}$. 

Since the diagram~\eqref{eq:cd-22} commutes and $H^2(f_1)$ is an
isomorphism, $H^2(f_1)$ restricts to an isomorphism
\[
v_1 \colon \ker H^2(\rho_1') \longisom \ker H^2(\rho_1)
\]
with $e_1 \circ v_1 = H^2(f_1) \circ g_1$.
It follows from Lemma~\ref{lem:Hirsch-iso} that the Hirsch extension
with $h$-invariant $e_1$ is isomorphic to that with $h$-invariant $g_1$,
so $\mcm_2$ is the Hirsch extension of $\mcm_1$ with $h$-invariant
$\ker H^2(\rho_1') \hookrightarrow H^2(\mcm_1)$.
Thus statement~\ref{cl:iii} holds for $n=1$.

\medskip
\noindent\textit{Inductive step.}
Assume statements~\ref{cl:i}, \ref{cl:ii}, and~\ref{cl:iii} hold for 
a given value of $n\ge 1$, and let 
$\rho_{n+1}^\prime \colon \mcm_{n+1} \to A$ be a structural
morphism. We need to show that the statements~\ref{cl:i}, 
\ref{cl:ii} and~\ref{cl:iii} then hold for $n+1$.

Consider statement \ref{cl:i}.
Given $\rho_{n+1}^\prime$, set $\rho_j^\prime$ equal to the
restriction of $\rho_{n+1}^\prime $ to $\mcm_j$ for $j \le n$.
By induction, from statement \ref{cl:i} it follows that
$((\mcm_j, d_j), \rho_j^\prime)_{1 \le j \le n}$ is an $n$th step
in the $1$-minimal model for $A$ and from statement \ref{cl:iii}
$\mcm_{n+1}$ is the Hirsch extension of $\mcm_n$ with
$h$-invariant 
$\ker H^2(\rho_n^\prime) \hookrightarrow H^2(\mcm_n)$.
Thus, $((\mcm_j, d_j), \rho_j^\prime)_{1 \le j \le n+1}$ is an
$(n+1)$th step in the $1$-minimal model for $A$, and statement
\ref{cl:i} holds for $n+1$. 

Consider statement~\ref{cl:ii}.
By statement~\ref{cl:ii} for $n$, there is a homotopy $\Phi_n$ between 
$\rho_n'$ and $\rho_n \circ f_n$.
Applying Lemma~\ref{lem:homotopy-lift} with 
$\varphi = \mathrm{id}_A$, $f_n$ as above, and $\Phi_n$,
there exists an isomorphism $f_{n+1} \colon \mcm_{n+1} \to \mcm_{n+1}$
extending $f_n$, and a homotopy $\Phi_{n+1}$ between 
$\rho_{n+1}^\prime$ and $\rho_{n+1} \circ f_{n+1}$.
Thus statement~\ref{cl:ii} holds for $n+1$.
 
Since $\rho_n' \simeq \rho_n \circ f_n$ by statement~\ref{cl:ii}
for $n$, and homotopic maps induce the same map on cohomology, 
we have $H^2(\rho_n') = H^2(\rho_n) \circ H^2(f_n)$, and since 
$f_n$ is an isomorphism, 
$\ker H^2(\rho_n') = H^2(f_n)^{-1}(\ker H^2(\rho_n))$.
Consider the diagram with exact rows
\begin{equation}
\label{eq:cd-2}
\begin{tikzcd}[column sep=30pt]
0 \ar[r] &[-6pt] H^1(\T_R(\bbX_{n+1}), d_{\mathbf{0}}) \ar[r, "e_n"]
    &[-2pt] H^2(\mcm_n) \ar[r, "H^2(\rho_n)"]
    &[4pt] H^2(A) \\
0 \ar[r] & \ker H^2(\rho_n')
    \ar[r, "g_n"]
    & H^2(\mcm_n)
        \ar[u, "\cong"', "H^2(f_n)"]
        \ar[r, "H^2(\rho_n')"]
    & H^2(A), \ar[u, "\id"]
\end{tikzcd}
\end{equation}
where $e_n$ is the $h$-invariant of the Hirsch extension
$\mcm_n \hookrightarrow \mcm_{n+1}$, and 
$g_n$ is the inclusion of $\ker H^2(\rho_n')$.
Since $\rho_n' \simeq \rho_n \circ f_n$, the diagram commutes, and
$H^2(f_n)$ induces an isomorphism
\[
\begin{tikzcd}[column sep=22pt]
v \colon  \ker H^2(\rho_n') \arrow[r, "\simeq"] & \ker H^2(\rho_n)
\end{tikzcd}
\]
with $e_n \circ v = H^2(f_n) \circ g_n$.
Hence, the diagram \eqref{eq:cd-2} with the map $v$ added commutes,
and it follows from Lemma \ref{lem:Hirsch-iso} that the Hirsch
extension with $h$-invariant $e_n$ is isomorphic to the Hirsch
extension with $h$-invariant $g_n$. Since the Hirsch extension
with $h$-invariant $e_n$ is $\mcm_{n+1}$, statement \ref{cl:iii}
holds for $n+1$, and the proof is complete.
\end{proof}

\subsection{Proof of the main theorem}
\label{subsec:proof-main}
The following theorem makes 
Theorem~\ref{thm:A} from the
Introduction precise.

\begin{theorem}
\label{thm:mnequalsquotient}
Let $R=\Z$ or $\F_p$ ($p$ prime), let $X$ be a connected $\Delta$-complex with
$H^1(X;R)$ finitely generated, and let 
$((\mcm_n,d_n), \rho_n)_{n \ge 1}$
be a $1$-minimal model for $A=C^*(X;R)$.
Set $(\mcm,d)=\varinjlim_n (\mcm_n,d_n)$, and set $G=\pi_1(X)$.
Then the towers $\{\GGR{G}{n+1}\}_{n\ge 1}$ and 
$\{G(\mcm_n)\}_{n\ge 1}$
are compatibly isomorphic: there exist isomorphisms $g_n\colon \GGR{G}{n+1} \to G(\mcm_n)$
such that the ladder
\[
\begin{tikzcd}[column sep=26pt]
\GGR{G}{2} \ar[d, "g_1"] & \GGR{G}{3} \ar[l, swap, "q_2"] \ar[d, "g_2"] & \cdots \ar[l]
& \GGR{G}{n+1} \ar[l] \ar[d, "g_n"] & \GGR{G}{n+2} \ar[l, swap, "q_{n+1}"] \ar[d, "g_{n+1}"] & \cdots \ar[l] \\
G(\mcm_1) & G(\mcm_2) \ar[l, swap, "\bar{\jmath}_1"] & \cdots \ar[l] & G(\mcm_n) \ar[l]
& G(\mcm_{n+1}) \ar[l, swap, "\bar{\jmath}_n"] & \cdots \ar[l]
\end{tikzcd}
\]
commutes.

Moreover, this isomorphism of towers is natural in $G=\pi_1(X)$:
given a basepoint-preserving map $\phi\colon X\to X'$ of connected
$\Delta$-complexes (with $G'=\pi_1(X')$, $1$-minimal model
$((\mcm',d'), \rho_n^\prime)$ for $C^*(X';R)$, and 
isomorphisms $g_n'$ for $X'$),
let $\phi_*^{(n)}\colon \GGR{G}{n+1}\to \GGR{G'}{n+1}$
be the map induced by $\phi_*\colon G\to G'$ on $R$-nilpotent quotients,
and let $\widehat\phi_n\colon G(\mcm_n)\to G(\mcm_n')$ be the group
homomorphism induced by the $1$-minimal model lift
$\widehat\phi\colon \mcm(A)\to\mcm(A')$ of
$\phi^*\colon C^*(X';R)\to C^*(X;R)$.
Then for each $n\ge 1$ the square
\begin{equation}
\label{dgm:naturality}
\begin{tikzcd}[column sep=36pt, row sep=24pt]
\GGR{G}{n+1} \ar[r, "\phi_*^{(n)}"] \ar[d, "g_n"', "\cong"]
  & \GGR{G'}{n+1} \ar[d, "\cong"', "g_n'"] \\
G(\mcm_n)       \ar[r, "\widehat\phi_n"']
  & G(\mcm_n')
\end{tikzcd}
\end{equation}
commutes, and the squares \eqref{dgm:naturality} are compatible with the
ladder maps $q_{n+1}$, $q_{n+1}'$, $\bar{\jmath}_n$, and $\bar{\jmath}_n'$.
\end{theorem}

\begin{proof}
To incorporate naturality, fix a basepoint-preserving map 
$\phi\colon X\to X'$
of connected $\Delta$-complexes, with $G'=\pi_1(X')$ and $1$-minimal model $((\mcm',d'), \rho_n^\prime)$ for $C^*(X';R)$. 
By Theorem~\ref{thm:1-min-lift}, the cochain
map $\phi^*\colon C^*(X';R)\to C^*(X;R)$ lifts to a morphism
$\widehat\phi\colon \mcm(A)\to\mcm(A')$, unique up to homotopy;
write $\widehat\phi_n$ for the induced homomorphism $G(\mcm_n)\to G(\mcm_n')$.

We prove by induction on $n\ge 1$ that the following four properties hold:
\begin{enumerate}[label=(\roman*), itemsep=2pt]
\item\label{ind:iso} $g_n\colon \GGR{G}{n+1} \isom G(\mcm_n)$ is an isomorphism;
\item\label{ind:ladder} 
$\bar{\jmath}_{n-1} \circ g_n = g_{n-1} \circ q_n$ for $n\ge 2$;
\item\label{ind:structural} the map
$\wht{\rho}_n$ defined by
$\wht{\rho}_n = h_{n+1}^\sharp \circ g_n^\sharp \circ \psi_n
      \colon \mcm_n \to C^*(X;R)$ is a structural morphism;
\item\label{ind:natural} the square \eqref{dgm:naturality} 
commutes, 
where $\phi_*^{(n)}\colon\GGR{G}{n+1}\to\GGR{G'}{n+1}$ is induced
by $\phi_*\colon G\to G'$ and $\widehat\phi_n\colon G(\mcm_n)\to G(\mcm_n')$
is induced by $\widehat\phi$.
\end{enumerate}

\medskip
\noindent\textit{Base case $(n=1)$.}
Since $H^1(X;R)$ is finitely generated, $\GGR{G}{2} = H_1(G;R)$
is a finitely generated free $R$-module. The $1$-minimal model 
gives $\mcm_1 = \T_R(\bbX_1)$ with $\bbX_1$ a basis for 
$H^1(X;R) \cong \Hom_R(\GGR{G}{2},R)$. The dual group 
$G(\mcm_1) = R^{\bbX_1}$ with multiplication $\mu(f,g)=f+g$
(since $d_1=0$ on generators). Setting $|\bbX_1|=k$, define 
$g_1\colon \GGR{G}{2} \to G(\mcm_1)$ by 
$g_1(a_1,\dots,a_k)\colon \bbX_1 \to R$, $x_i \mapsto a_i$. 
This is an isomorphism of abelian groups, giving~\ref{ind:iso}. 
Property~\ref{ind:ladder} is vacuous.
Each of the maps $\psi_1$, $g_1^\#$, and $h_2^\#$ in the diagram
\[
\begin{tikzcd}[column sep=24pt, row sep=26pt]
\mcm_{1}\ar[r, "\psi_{1}" ]
& C^\ast(K(G(\mcm_{1}),1);R) \ar[r, "g_{1}^{\#}" ]
& C^\ast(K(\GGR{G}{2},1);R) \ar[r, "h_{2}^{\#}" ]
& C^\ast(X;R)
\end{tikzcd}
\] 
induces isomorphisms on $H^0$ and $H^1$.
Thus, the composition
$\wht{\rho}_1 = h_2^\# \circ g_1^\# \circ \psi_1$ 
induces isomorphisms on $H^0$ and $H^1$, and it follows that
$\wht{\rho}_1 $ is a structural morphism giving~\ref{ind:structural}.

For~\ref{ind:natural}, both $G(\mcm_1)=R^{\bbX_1}$ and $\GGR{G}{2}=H_1(G;R)$ are
identified with $\Hom_R(H^1(X;R),R)$ via $g_1$, and likewise for $X'$.
The map $\phi^*\colon H^1(X';R)\to H^1(X;R)$ induces both
$\phi_*^{(1)}\colon H_1(G;R)\to H_1(G';R)$ and
$\widehat\phi_1\colon G(\mcm_1)\to G(\mcm_1')$
as the dual of $\phi^*$ on $H^1$, so $g_1'\circ \phi_*^{(1)}=\widehat\phi_1\circ g_1$.

\medskip
\noindent\textit{Inductive step.}
Assume properties~\ref{ind:iso}--\ref{ind:natural} hold for $n$.
Property~\ref{ind:structural} gives the commutative diagram
\[
\begin{tikzcd}[column sep=24pt, row sep=26pt]
\mcm_{n}\ar[r, "\psi_{n}" ]
\ar[rrr, bend right=15, "\wht{\rho}_n" ]
& C^\ast(K(G(\mcm_{n}),1);R) \ar[r, "g_{n}^{\#}" ]
& C^\ast(K(\GGR{G}{n+1},1);R) \ar[r, "h_{n+1}^{\#}" ]
& C^\ast(X;R)
\end{tikzcd}
\]
inducing on cohomology
\[
\begin{tikzcd}[column sep=36pt]
0 \ar[r] & H^1(\T_R(\bbX_{n+1}), d_\bz) \ar[r, "e_n"]
\ar[d, "\cong" '] & H^2(\mcm_{n})
\ar[r, "H^2(\wht{\rho}_{n})"]
\ar[d, "H^2(g_{n}^{\#} \circ \psi_n)", "\cong" '] & H^2(X;R) \phantom{.}
\ar[d, "\id", "\cong" '] \\
0 \ar[r] & H^1(\gr_{n+1}^R (G);R) \ar[r, "f_{n+1}"]
& H^2(\GGR{G}{n+1};R)
\ar[r, "H^2(h_{n+1}^{\#})"]
& H^2(X;R) .
\end{tikzcd}
\]
The top row is exact by Lemma~\ref{lem:classify} 
since $\wht{\rho}_n$ is a structural morphism and
$\mcm_{n+1}$ is the Hirsch extension of $\mcm_n$
with $h$-invariant  $\ker H^2(\wht{\rho}_n) \hookrightarrow H^2(\mcm_n)$.  
The bottom row is exact by Lemma~\ref{lem:3-term-seq}.  
Commutativity follows from 
$\wht{\rho}_n =  h_{n+1}^\sharp \circ g_n^\sharp \circ \psi_n$.  
The vertical map on $H^2$ is an isomorphism by 
Lemma~\ref{lem:Hirsch-central} and~\ref{ind:iso}.

Thus $e_n$ and $f_{n+1}$ define isomorphic central extensions: $G(\mcm_{n+1})$
of $G(\mcm_n)$ and $\GGR{G}{n+2}$ of $\GGR{G}{n+1}$.  
By Lemma~\ref{lem:eG1G2}, $g_n$ extends uniquely to an isomorphism
$g_{n+1}\colon \GGR{G}{n+2} \to G(\mcm_{n+1})$ with
$\bar{\jmath}_n \circ g_{n+1} = g_n \circ q_{n+1}$,
giving~\ref{ind:iso} and~\ref{ind:ladder} at level $n+1$.

For~\ref{ind:structural} at level $n+1$, the diagram
\[
\begin{tikzcd}[column sep=22pt, row sep=28pt]
\mcm_{n+1}\ar[r, "\psi_{n+1}" '] 
\ar[rrr, bend left=14, "\wht{\rho}_{n+1}" ']
& C^\ast(K(G(\mcm_{n+1}),1);R) \ar[r, "g_{n+1}^{\#}" ']
& C^\ast(K(\GGR{G}{n+2},1);R) \ar[r, "h_{n+2}^{\#}" ']
& C^\ast(X;R) \\
\mcm_{n}\ar[r, "\psi_{n}" ] \ar[u , "{j_n}" '] 
\ar[rrr, bend right=14, "\wht{\rho}_n" ]
& C^\ast(K(G(\mcm_{n}),1);R) \ar[r, "g_{n}^{\#}" ] \ar[u , "{\bar{\jmath}_{n}}^{\,\sharp}" ']
& C^\ast(K(\GGR{G}{n+1},1);R) \ar[r, "h_{n+1}^{\#}" ] \ar[u, "{q}_{n+1}^{\#}" ]
& C^\ast(X;R) \ar[u, "\id" ']
\end{tikzcd}
\]
commutes by Lemma~\ref{lem:Hirsch-central}, the extension property,
and the Postnikov tower \eqref{eq:postnikov-3}.  
Since $j_{1,n+1} = j_n \circ j_{1,n}$, we have 
\[
\wht{\rho}_{n+1} \circ j_{1,n+1} 
= \wht{\rho}_{n+1} \circ j_n \circ j_{1,n}
= \wht{\rho}_n \circ j_{1,n} = \wht{\rho}_1,
\]
which induces isomorphisms on $H^0$ and $H^1$ by the base case.
Hence $\wht{\rho}_{n+1}$ is a structural morphism.

For~\ref{ind:natural} at level $n+1$, consider the commutative diagram on cohomology
\[
\begin{tikzcd}[column sep=30pt, row sep=26pt]
0 \ar[r] & H^1(\T_R(\bbX_{n+1}), d_\bz) \ar[r, "e_n"] \ar[d]
  & H^2(\mcm_{n}) 
  \ar[r, "H^2(\wht{\rho}_n)"] 
  \ar[d, "H^2(\widehat\phi_n^{\,\sharp})"']
  & H^2(X;R) \ar[d, "H^2(\phi^*)"] \\
0 \ar[r] & H^1(\T_R(\bbX_{n+1}'), d_\bz) \ar[r, "e_n'"]
  & H^2(\mcm_{n}') 
  \ar[r, "H^2(\wht{\rho}_n')"]
  & H^2(X';R),
\end{tikzcd}
\]
where the left vertical arrow is the restriction of $H^2(\widehat\phi_n^\sharp)$
to the respective kernels. Commutativity of the right square follows from
the naturality of the $1$-minimal model lift (Theorem~\ref{thm:1-min-lift}).
Commutativity of the left square follows from~\ref{ind:natural} at level $n$,
together with the functoriality of $\psi_n$ recorded in
Lemma~\ref{lem:Hirsch-central}\eqref{pt:four}.
Now both composites $g_{n+1}'\circ \phi_*^{(n+1)}$ and
$\widehat\phi_{n+1}\circ g_{n+1}$ are extensions of
$g_n'\circ\phi_*^{(n)} = \widehat\phi_n\circ g_n$
(by~\ref{ind:natural} at level $n$) to the central extensions
$\GGR{G}{n+2}$ and $G(\mcm_{n+1})$.
Since such an extension is unique by Lemma~\ref{lem:eG1G2},
diagram~\eqref{dgm:naturality} commutes at level $n+1$.
This completes the induction.
\end{proof}

\begin{corollary}
\label{cor:postnikov-hirsch}
With notation and assumptions as above, the following hold:
\begin{enumerate}[itemsep=2pt]
\item The space $K(\GGR{G}{k+1},1)$ in the Postnikov system for $X$ 
is the bar construction applied to the group $G(\mcm_{k}(X))$.
\item For the map $h_k \colon X \to K(\GGR{G}{k},1)$, we have that
$\mcm_{k+1}(X)$ is the Hirsch extension given by the
inclusion of $\ker H^2(h_k)$ in $H^2(\mcm_k(X))$.
\end{enumerate}
\end{corollary}

\begin{proof}
Both statements follow directly from 
Theorem~\ref{thm:mnequalsquotient} and 
Lemma~\ref{lem:classify}\ref{cl:iii}, together with the 
identification of the Postnikov stages with bar constructions
in Remark~\ref{rem:BG-vs-KG1}.
\end{proof}


\section{Strict representability and classifying maps}
\label{sect:functorial}

The central result of this section, Theorem~\ref{thm:strict-representability},
establishes a natural bijection between dga maps out of 
a colimit of Hirsch extensions $\mcm$ and $\Delta$-set maps into 
the bar construction $BG(\mcm)$.  Informally, this says that $\mcm$ 
\emph{represents} the functor $Y \mapsto \Hom_{\Deltaset}(Y, BG(\mcm))$ 
on the category of $\Delta$-sets, with cochains as the representing object.
We prove the bijection is natural in both $Y$ and $\mcm$.

We begin with two key lemmas (\S\ref{subsec:key-lemmas})
and conclude with the representability theorem and the
correspondence between classifying maps and structural morphisms
(\S\S\ref{subsec:strict-rep}--\ref{subsec:classifying-structural}).
Background on $\Delta$-sets, including face operators on $BG$ 
used in the proofs, is collected in Appendix~\ref{subsec:delta}.

\subsection{Algebra maps, group labelings, and the differential condition}
\label{subsec:key-lemmas}

For a topological space $Z$ and a topological group $H$, we write
$\Map_{\cts}(Z, H)$ for the set of continuous maps from $Z$ to $H$.
When $H = G(\mcm) = \varprojlim_n G(\mcm_n)$ carries the inverse-limit
topology, the identification
\begin{equation}
\label{eq:map-cts-gm}
\Map_{\cts}(Z,\, G(\mcm)) = \varprojlim_n\, \Map_{\cts}(Z,\, G(\mcm_n))
\end{equation}
holds for any $Z$, since a continuous map into an inverse limit of
topological groups is the same as a compatible family of continuous
maps into the terms.

In the remainder of this section, $Z = Y_1$ is the set of
$1$-simplices of a $\Delta$-set $Y$, equipped with the discrete
topology; since every map from a discrete space is continuous,
\[
\Map_{\mathrm{cts}}(Y_1, G(\mcm)) = \Map(Y_1, G(\mcm)),
\]
and we write simply $\Map(Y_1, G(\mcm))$ throughout.

The two lemmas of this subsection are the technical heart of the
representability theorem.  The first identifies morphisms of
$R$-binomial cup-one graded algebras with labelings of the
$1$-simplices of a $\Delta$-set by elements of $G(\mcm)$.
The second characterizes, in purely group-theoretic terms, when such
a labeling extends to a dga morphism.  Taken together, they reduce
the representability problem to a condition that can be checked
simplex-by-simplex.

\begin{lemma}
\label{lem:Delta1}
Let $\mcm = \T_R(\bbX)$ be a colimit of Hirsch extensions,
and let $Y$ be a $\Delta$-set.
Then there is a bijection, natural in $Y$,
\begin{equation}
\label{eq:bij-Delta1}
r \colon \Hom_{\bincupalg}(\mcm,\, C^\ast(|Y|;R))
\longisom \Map(Y_1,\, G(\mcm)),
\end{equation}
between morphisms $f \colon \mcm \to C^\ast(|Y|;R)$ of $R$-binomial
cup-one graded algebras and maps of sets $Y_1 \to G(\mcm)$.
Naturality means that for every map of $\Delta$-sets $t \colon Y \to Z$,
the diagram
\begin{equation}
\label{dgm:natural-Delta1}
\begin{tikzcd}[row sep=24pt, column sep=48pt]
\Hom(\mcm,\, C^\ast(|Z|;R))
    \ar[r, "r"]
    \ar[d, "{(|t|^\#)_*}"']
& \Map(Z_1,\, G(\mcm))
    \ar[d, "{(t_1)^*}"] \\
\Hom(\mcm,\, C^\ast(|Y|;R))
    \ar[r, "r"]
& \Map(Y_1,\, G(\mcm))
\end{tikzcd}
\end{equation}
commutes, where $(|t|^\#)_*$ denotes precomposition with
$|t|^\# \colon C^\ast(|Z|;R)\to C^\ast(|Y|;R)$,
and $(t_1)^*$ denotes precomposition with $t_1\colon Y_1\to Z_1$.
\end{lemma}

\begin{proof}
Recall that $G(\mcm) = M(\bbX;R)$ is the $R$-module of functions
$\bbX \to R$.  We construct mutually inverse maps between the two
sets, then verify naturality.

\smallskip\noindent\textit{Step 1: From algebra maps to set maps.}
Given $f \colon \mcm \to C^\ast(|Y|;R)$, define
$r(f) \colon Y_1 \to G(\mcm)$ by setting, for each $y \in Y_1$ and
$x \in \bbX$,
\[
  r(f)(y)(x) = f(x)\big|_{|y|} \in R,
\]
where $f(x) \in C^1(|Y|;R)$ is evaluated on the image of $(0,1)$
under the characteristic map $|y| \colon |\Delta^1| \to |Y|$.
The assignment $x \mapsto f(x)|_{|y|}$ defines an element of
$G(\mcm)$, so $r(f)$ is a well-defined map of sets.

\smallskip\noindent\textit{Step 2: From set maps to algebra maps.}
Given $g \colon Y_1 \to G(\mcm)$, define
$f_g \colon \mcm \to C^\ast(|Y|;R)$ as follows.
On degree $0$, send the unit of $\mcm^0$ to the unit cochain.
On degree $1$, since $\mcm$ is freely generated as an $R$-binomial
cup-one graded algebra by $\bbX \subset \mcm^1$, any such map is
uniquely determined by its values on $\bbX$; define
$f_g(x) \in C^1(|Y|;R)$ by $f_g(x)|_{|y|} = g(y)(x)$ for each
$y \in Y_1$.  By freeness this extends uniquely to a morphism of
$R$-binomial cup-one graded algebras.

\smallskip\noindent\textit{Step 3: Mutual inverses.}
For any $f$: $f_{r(f)}(x)|_{|y|} = r(f)(y)(x) = f(x)|_{|y|}$,
so $f_{r(f)} = f$.
For any $g$: $r(f_g)(y)(x) = f_g(x)|_{|y|} = g(y)(x)$,
so $r(f_g) = g$.

\smallskip\noindent\textit{Step 4: Naturality.}
For $t\colon Y\to Z$, $f\colon \mcm\to C^\ast(|Z|;R)$,
$y\in Y_1$, and $x\in\bbX$:
\[
  r(|t|^\#\circ f)(y)(x) = f(x)\big|_{|t(y)|}
  = r(f)(t_1(y))(x) = (r(f)\circ t_1)(y)(x),
\]
where the first equality uses that $|t|^\#$ pulls back along $|t|$,
and the second that $|t|$ maps the characteristic simplex of $y$ to
that of $t_1(y)$.
\end{proof}

The next lemma shows that the differential condition on $f$ is
equivalent, via the labeling $r(f)$, to a multiplicativity condition
on the group $G(\mcm)$.

\begin{lemma}
\label{lem:rfygroup}
Let $\mcm = \T_R(\bbX)$ be a colimit of Hirsch extensions,
let $f \colon \mcm \to C^\ast(\Delta^2;R)$ be a morphism of
$R$-binomial cup-one graded algebras,
and let $r(f)$ be the labeling of $1$-simplices in $\Delta^2$
by elements of $G(\mcm)$ defined above.
Then $f$ commutes with the differentials if and only if
\begin{equation}
\label{eq:rfygroup}
r(f)(0,1) \cdot r(f)(1,2) = r(f)(0,2),
\end{equation}
where $\,\cdot\,$ denotes the product in $G(\mcm)$.
\end{lemma}

\begin{proof}
Set $\ba = r(f)(0,1)$, $\bb = r(f)(1,2)$, $\bc = r(f)(0,2)$
in $R^\ell$.
For each generator $x_i \in \bbX$ with
$d_\mcm x_i = \sum_j p_{x_i,j} \otimes q_{x_i,j}$,
the commutativity condition $f\circ d_\mcm(x_i) = d_\Delta\circ f(x_i)$
reads on the left:
\[
f\circ d_\mcm(x_i)
  = \sum_j f(p_{x_i,j})\cup f(q_{x_i,j})
  = \sum_j p_{x_i,j}(\ba)\,q_{x_i,j}(\bb),
\]
where the last equality uses the Alexander--Whitney formula
for the cup product (Appendix~\ref{subsec:delta},
equation~\eqref{eq:cup-product}) applied to the unique nontrivial
pairing on $\Delta^2$.
On the right, since $d_\Delta(u) = u(0,1)+u(1,2)-u(0,2)$ for
$u\in C^1(\Delta^2;R)$:
\[
d_\Delta\circ f(x_i) = a_i + b_i - c_i.
\]
Equating these gives $c_i = a_i + b_i - \sum_j p_{x_i,j}(\ba)\,q_{x_i,j}(\bb)$,
which is precisely the multiplication formula \eqref{eq:product}
for $\ba\cdot\bb = \bc$ in $G(\mcm)$.
\end{proof}

\begin{remark}
\label{rem:map-label}
Combining the two lemmas: a morphism $\mcm \to C^\ast(|Y|;R)$ of
graded algebras corresponds to a labeling of the $1$-simplices of
$|Y|$ by elements of $G(\mcm)$, and it is a dga morphism if and only
if, for every $2$-simplex of $|Y|$, the product of the labels on the
two forward faces equals the label on the remaining face.
This is the local, simplex-by-simplex content of the global bijection
proved in Theorem~\ref{thm:strict-representability} below.
\end{remark}

\subsection{Strict representability}
\label{subsec:strict-rep}

We now assemble the preceding lemmas into the main result of this
section.  The theorem says that there is a bijection between
$\Hom_{\Deltaset}\bigl(Y,\, BG(\mcm)\bigr)$ and
$\Hom_{\bincupdga}\bigl(\mcm,\, C^*(|Y|;R)\bigr)$, natural in the
$\Delta$-set $Y$, with the representing isomorphism given explicitly
by the map $r$ of Definition~\ref{def:rf-y}.
In categorical terms, the functor
$Y \mapsto \Hom_{\bincupdga}(\mcm, C^*(|Y|;R))$
on $\Delta$-sets is represented by $BG(\mcm)$.

\begin{definition}
\label{def:rf-y}
Given a colimit of Hirsch extensions $\mcm$, a $\Delta$-set
$Y$, and a morphism $f \colon \mcm \to C^\ast(|Y|;R)$ of $R$-binomial
cup-one graded algebras, define maps
$r(f)_n\colon Y_n \to BG(\mcm)_n$ for $n \ge 1$ by
\begin{equation}
\label{eq:rf-y}
r(f)_n(y) = [r(f)(y)(0,1)\mid r(f)(y)(1,2)\mid
             \cdots \mid r(f)(y)(n-1,n)],
\end{equation}
where $r(f)(y)(j,j+1)$ is the element of $G(\mcm)$ mapping $x_i$
to $f(x_i)$ evaluated on the image of $(j,j+1)$ in $|Y|$.
\end{definition}

\begin{theorem}
\label{thm:strict-representability}
Let $\mcm = \T_R(\bbX)$ be a colimit of Hirsch extensions,
and let $Y$ be a $\Delta$-set. There is a bijection, natural in both
$Y$ and $\mcm$,
\begin{equation}
\label{eq:strict-rep}
\Hom_{\Deltaset}(Y,\, BG(\mcm)) \cong
\Hom_{\bincupdga}(\mcm,\, C^\ast(|Y|;R)).
\end{equation}
Under this bijection a morphism $f \colon \mcm \to C^\ast(|Y|;R)$
corresponds to the map $r(f) \colon Y \to BG(\mcm)$ of
Definition~\ref{def:rf-y}, and the triangle
\begin{equation}
\label{dgm:adjunction}
\begin{tikzcd}[row sep=28pt, column sep=38pt]
  & \mcm \ar[d, "\psi"'] \ar[dl, "f"'] \\
C^\ast(|Y|;R) & C^\ast(|BG(\mcm)|;R) \ar[l, "{|r(f)|^{\#}}"]
\end{tikzcd}
\end{equation}
commutes.  Naturality in $Y$ means that for every map
$t \colon Y \to Z$ of $\Delta$-sets the square
\begin{equation}
\label{dgm:nat-Y}
\begin{tikzcd}[row sep=24pt, column sep=52pt]
\Hom_{\Deltaset}(Z,\,BG(\mcm)) \ar[r,"\cong"] \ar[d,"t^*"']
  & \Hom_{\bincupdga}(\mcm,\,C^*(|Z|;R)) \ar[d,"{(|t|^\#)_*}"] \\
\Hom_{\Deltaset}(Y,\,BG(\mcm)) \ar[r,"\cong"]
  & \Hom_{\bincupdga}(\mcm,\,C^*(|Y|;R))
\end{tikzcd}
\end{equation}
commutes.  Naturality in $\mcm$ means that for every morphism
$\alpha \colon \mcm \to \mcm'$ of binomial $\cup_1$-dgas the square
\begin{equation}
\label{dgm:nat-mcm}
\begin{tikzcd}[row sep=24pt, column sep=52pt]
\Hom_{\Deltaset}(Y,\,BG(\mcm')) \ar[r,"\cong"] \ar[d,"(B\alpha_*)^*"']
  & \Hom_{\bincupdga}(\mcm',\,C^*(|Y|;R)) \ar[d,"\alpha^*"] \\
\Hom_{\Deltaset}(Y,\,BG(\mcm)) \ar[r,"\cong"]
  & \Hom_{\bincupdga}(\mcm,\,C^*(|Y|;R))
\end{tikzcd}
\end{equation}
commutes, where $\alpha_*\colon G(\mcm)\to G(\mcm')$ is induced by
$\alpha$ and $B\alpha_*$ is its classifying map.
\end{theorem}

\begin{proof}
\noindent\textit{Step 1: $r(f)$ is a map of $\Delta$-sets.}
We check $r(f)_{n-1}(d_i y) = d_i(r(f)_n(y))$ for all $y\in Y_n$.
For $1\le i\le n-1$, the face operator $d_i$ on $BG(\mcm)$ multiplies
the $i$-th and $(i+1)$-th entries; on $Y$ it omits vertex $i$.
Applying Lemma~\ref{lem:rfygroup} to the $2$-simplex spanned by
vertices $i-1$, $i$, $i+1$ in the image of $y$ gives exactly the
group multiplication condition \eqref{eq:rfygroup} needed for
commutativity.  The cases $i=0$ and $i=n$ follow directly from
the definitions \eqref{eq:faceBG} and \eqref{eq:faceDelta}.

\medskip\noindent\textit{Step 2: The triangle \eqref{dgm:adjunction} commutes.}
Since $\mcm$ is freely generated in degree $1$, it suffices to check
on generators.  For $y \in Y_1$, set $b_i = f(x_i)|_{|y|}$, so
$r(f)(y) = (b_1,\ldots,b_\ell)$.  By definition of $\psi$, we have
$\psi(x_i)|_{|(b_1,\ldots,b_\ell)|} = b_i$.  Thus
$|r(f)|^\#\circ\psi(x_i)$ and $f(x_i)$ agree on every $1$-simplex,
giving commutativity.

\medskip\noindent\textit{Step 3: Construction of the inverse.}
Given $g\colon Y\to BG(\mcm)$, define $f_g$ via Lemma~\ref{lem:Delta1}
from $g_1 = g|_{Y_1}$.  For any $2$-simplex $y\in Y_2$, writing
$g_2(y) = [u\mid v]$, the $\Delta$-set condition on $g$ gives
$g_1(d_2 y) = u$, $g_1(d_0 y) = v$, $g_1(d_1 y) = uv$,
so the labels satisfy $g_1(d_2 y)\cdot g_1(d_0 y) = g_1(d_1 y)$.
By Lemma~\ref{lem:rfygroup} this is exactly the condition for $f_g$
to commute with the differential; since $y$ was arbitrary, $f_g$ is
a dga morphism.  The identities $f_{r(f)} = f$ and $r(f_g) = g$
from Lemma~\ref{lem:Delta1} show the two maps are mutually inverse.

\medskip\noindent\textit{Step 4: Naturality in $Y$.}
This follows from Lemma~\ref{lem:Delta1}: the maps $f\mapsto r(f)$ and
$g\mapsto f_g$ are defined by formulas independent of $Y$.

\medskip\noindent\textit{Step 5: Naturality in $\mcm$.}
For $\alpha\colon\mcm\to\mcm'$ and $f'\colon\mcm'\to C^*(|Y|;R)$,
the induced map satisfies
$r(f'\circ\alpha)(y)(x) = f'(\alpha(x))|_{|y|} = r(f')(y)(\alpha(x))
= (\alpha_*\circ r(f'))(y)(x)$,
which is precomposition by $\alpha$ on the dga side and postcomposition
by $\alpha_*$ (hence precomposition by $B\alpha_*$) on the group side.
This gives the commutativity of \eqref{dgm:nat-mcm}.
\end{proof}

\subsection{Classifying maps and structural morphisms}
\label{subsec:classifying-structural}

In this section we define classifying maps 
and then show that the bijection in Theorem 
\ref{thm:strict-representability} restricts to a bijection
between classifying maps $f\colon Y \to K(N,1)$ and 
structural morphisms from $\mcm(N)$ to $C^\ast(Y;R)$.

\begin{definition}
\label{def:classifying}
Given $Y$, let $G=\pi_1(Y)$ and consider the Postnikov tower
of fibrations of the spaces $K(\GGR{G}{n},1)$.
A map $f_n \colon Y \to K(\GGR{G}{n},1)$ is called a 
\textit{classifying map} if the diagram
\begin{equation}
\label{dgm:cm}
\begin{tikzcd}
&  K(\GGR{G}{n},1) \ar[d,  "p_{n,1}" ]\\
Y \ar[ur, "f_n" ] \ar[r, "f_1" '] & K(\GGR{G}{2},1)	
\end{tikzcd}
\end{equation}
commutes with 
$H^1(f_1)\colon H^1(K(\GGR{G}{2},1);R) \to H^1(Y;R)$
an isomorphism.
\end{definition}

\begin{lemma}
\label{lem:correspondence}
Let $N$ be a nilpotent group with $\Gamma_n^R(N) = 1$ (equivalently,
$N = \GGR{N}{n}$), let $\mcm(N)$ be the $1$-minimal model
for $C^\ast(K(N,1);R)$, and let $Y$ be a $\Delta$-complex.
Then the bijection of Theorem~\ref{thm:strict-representability} 
restricts to a bijection between classifying maps $f\colon Y \to K(N,1)$ 
and structural morphisms $s\colon\mcm(N) \to C^\ast(Y;R)$.
\end{lemma}

\begin{proof}
Recall that $G(\mcm(N))=N$ (since $\mcm(N)$ is the $1$-minimal model
for $C^*(K(N,1);R)$ and $N$ is nilpotent), so $K(N,1)=BG(\mcm(N))$
as $\Delta$-sets.  Write $\mcm = \mcm(N)$ and let
\begin{equation}
\label{eq:Phi-map}
\begin{tikzcd}[column sep=22pt]
\Phi \colon  \Hom_{\bincupdga}\bigl(\mcm,\, C^\ast(|Y|;R)\bigr)
  \arrow[r, "\simeq"] &
  \Hom_{\Deltaset}\bigl(Y,\, K(N,1)\bigr)
\end{tikzcd}
\end{equation}
denote the bijection of Theorem~\ref{thm:strict-representability}
(with $\Phi(s)=r(s)$, the map of Definition~\ref{def:rf-y}).
We show that under $\Phi$, structural morphisms correspond exactly
to classifying maps.

\medskip
\noindent\textit{Step 1: Notation for the Hirsch filtration.}
Write $\mcm = \varinjlim_k \mcm_k$ with $\mcm_1 = \T_R(\bbX_1)$
and inclusions $i_{1,k}\colon\mcm_1\hookrightarrow\mcm_k$.
By Lemma~\ref{lem:Hirsch-central}\eqref{pt:four}, the natural map
$\psi\colon\mcm\to C^*(K(N,1);R)$ satisfies
$\psi \circ i_{1,k} = p_{k,1}^\sharp \circ \psi_1$,
where $p_{k,1}\colon K(N,1)\to K(N/\Gamma_{k+1}(N),1)$ is the
Postnikov projection and $\psi_1\colon\mcm_1\to C^*(K(N/\Gamma_2(N),1);R)$
is the first-stage quasi-isomorphism.
In particular, the Postnikov projection $p_{n-1,1}\colon K(N,1)\to K(N^{\ab}_R,1)$
(corresponding to $k=1$) is identified on the dga side
with the inclusion $i_{1,1} = \id_{\mcm_1}$ composed with~$\psi_1$.

\medskip
\noindent\textit{Step 2: Structural morphisms map to classifying maps.}
Since $\Gamma_n^R(N)=1$, the $1$-minimal model stabilizes at stage
$n-1$: $\mcm=\mcm_{n-1}$. We apply Definition~\ref{def:structural-morphism}
at this level.
Let $s\colon\mcm\to C^\ast(|Y|;R)$ be a structural morphism,
so there exists $s_1\colon\mcm_1\to C^\ast(|Y|;R)$ with $H^1(s_1)$
an isomorphism and $s \circ i_{1,n-1} = s_1$. Set $f = r(s)\colon Y\to K(N,1)$. 
We verify that $f$ is a classifying map, i.e., that diagram~\eqref{dgm:cm}
commutes and $H^1(f_1)$ is an isomorphism.

By the naturality in $\mcm$ of $\Phi$ (diagram~\eqref{dgm:nat-mcm}),
applied to the inclusion $\alpha = i_{1,n-1}\colon\mcm_1\inj \mcm$,
we have
\begin{equation}
\label{eq:nat-alpha}
r(s \circ i_{1,n-1}) = B(\bar{\jmath}_{1,n-1})_\ast \circ r(s),
\end{equation}
where $\bar{\jmath}_{1,n-1}\colon G(\mcm)\to G(\mcm_1)$ is the
projection induced by $i_{1,n-1}$, and
$B(\bar{\jmath}_{1,n-1})\colon K(N,1)\to K(N/\Gamma_2(N),1)$
is the induced map of classifying spaces, which is precisely $p_{n-1,1}$.
Since $s\circ i_{1,n-1} = s_1$, equation~\eqref{eq:nat-alpha} gives
\[
r(s_1) = p_{n-1,1} \circ r(s) = p_{n-1,1} \circ f,
\]
so with $f_1 = r(s_1)\colon Y\to K(N/\Gamma_2(N),1) = K(N^{\ab}_R,1)$,
diagram~\eqref{dgm:cm} commutes.

It remains to check that $H^1(f_1)$ is an isomorphism.
Since $s_1$ is a structural morphism for $\mcm_1$ and
$\mcm_1\xrightarrow{\psi_1} C^*(K(N^{\ab}_R,1);R)$ is a
quasi-isomorphism (Theorem~\ref{thm:pronilpotent-model}),
the adjunction triangle~\eqref{dgm:adjunction} at stage $1$ gives
\[
s_1 = f_1^\sharp \circ \psi_1.
\]
Hence $H^1(s_1) = H^1(f_1)^\sharp \circ H^1(\psi_1)$.
Since both $H^1(s_1)$ and $H^1(\psi_1)$ are isomorphisms,
so is $H^1(f_1)$.  Thus $f = r(s)$ is a classifying map.

\medskip
\noindent\textit{Step 3: Classifying maps map to structural morphisms.}
Let $f\colon Y\to K(N,1)$ be a classifying map.
Set $s = f^\sharp \circ \psi\colon\mcm\to C^\ast(|Y|;R)$,
where $\psi\colon\mcm\to C^*(K(N,1);R)$ is the quasi-isomorphism
of Theorem~\ref{thm:pronilpotent-model}.
By the adjunction triangle~\eqref{dgm:adjunction}, $r(s)=f$,
so $s$ is the preimage of $f$ under~$\Phi$.
We verify that $s$ is a structural morphism.

Set $s_1 = s\circ i_{1,n-1} = (f\circ p_{n-1,1})^\sharp\circ\psi_1$.
Since $f$ is a classifying map, diagram~\eqref{dgm:cm} gives
$f\circ p_{n-1,1} = f_1\colon Y\to K(N^{\ab}_R,1)$,
so $s_1 = f_1^\sharp\circ\psi_1$.
Then $H^1(s_1) = H^1(f_1)^\sharp\circ H^1(\psi_1)$,
and both factors are isomorphisms (by hypothesis on $f_1$ and by
Theorem~\ref{thm:pronilpotent-model}), so $H^1(s_1)$ is an isomorphism.
Moreover, $s\circ i_{1,n-1} = s_1$ by construction,
so diagram~\eqref{eq:classify} commutes.  Thus $s$ is a structural morphism.

\medskip
\noindent\textit{Step 4: The two assignments are mutually inverse.}
Steps~2 and~3 show that the bijection $\Phi$ sends structural morphisms
to classifying maps and that the inverse bijection $\Phi^{-1}$ sends
classifying maps to structural morphisms.
Since $\Phi$ is already a bijection (Theorem~\ref{thm:strict-representability}),
these two restricted maps are mutually inverse bijections.
\end{proof}

\begin{remark}
\label{rem:GmcmN}
The proof uses the identification $G(\mcm(N))=N$.
This holds because $\mcm(N)$ is by definition the $1$-minimal model
for $C^*(K(N,1);R)$, and Theorem~\ref{thm:pronilpotent-model}
identifies the dual group of any $1$-minimal model with the
(pro)nilpotent group it models; since $N$ is already nilpotent
the inverse limit stabilizes and gives $G(\mcm(N))=N$.
\end{remark}

\begin{remark}
\label{rem:adjunction-triangle}
Step~3 uses the implication
$s = f^\sharp\circ\psi \Rightarrow r(s)=f$.
This is immediate from the adjunction triangle~\eqref{dgm:adjunction}:
since $\Phi(s)=r(s)$ and the triangle gives $r(g^\sharp\circ\psi)=g$
for every classifying map $g$, taking $g=f$ yields
$r(f^\sharp\circ\psi)=f$ directly.
The point requiring care is that~\eqref{dgm:adjunction} is asserted for
maps into a \emph{general} $\Delta$-complex $|Y|$, not only for
$Y=K(N,1)$; this is part of the statement of
Theorem~\ref{thm:strict-representability}, which is natural in $Y$.
\end{remark}

\section{Isomorphisms between nilpotent quotients}
\label{sec:extend}

This section has two parts. Theorems~\ref{thm:triangle}
and~\ref{thm:trianglespaces} give a necessary and sufficient condition
for an isomorphism $\GGR{G_a}{n}\cong\GGR{G_b}{n}$ between nilpotent
quotients to extend one step further, to
$\GGR{G_a}{n+1}\cong\GGR{G_b}{n+1}$: the first works algebraically, at
the level of $1$-minimal models, the second topologically, in terms of
classifying maps into $K(N,1)$. Theorem~\ref{thm:coker-torsion} then
identifies the integral invariant underlying this condition: the
cokernel of the structural morphism in degree~$2$ recovers the torsion
of the next graded piece $\gr_{n+1}$ of the lower central series,
complementing the torsion-free part recovered by its kernel through
Theorem~\ref{thm:mnequalsquotient}.

\subsection{The triangle theorems}
\label{subsec:triangle}

Let $X_a$ and $X_b$ be path-connected $\Delta$-complexes with
fundamental groups $G_a = \pi_1(X_a)$ and $G_b = \pi_1(X_b)$,
and assume $H^1(X_a; R)$ and $H^1(X_b; R)$ are finitely generated.
We address the following question: given an isomorphism
$\GGR{G_a}{n} \cong \GGR{G_b}{n}$, under what conditions 
is there an isomorphism $\GGR{G_a}{n+1} \cong \GGR{G_b}{n+1}$?

The condition of Theorem~\ref{thm:trianglespaces}, phrased in terms of
classifying maps into $K(N,1)$, is deduced from the purely algebraic
Theorem~\ref{thm:triangle}, which works at the level of $1$-minimal
models, via the correspondence given by
Theorems~\ref{thm:mnequalsquotient} and~\ref{thm:strict-representability}
and Lemma~\ref{lem:correspondence}.

Given a morphism $f\colon A\to B$ of cochain algebras, 
we write $\im(H^i(f)) \subset H^i(B)$ for the image of the 
induced map on $i$th cohomology.

\begin{theorem}
\label{thm:triangle}
Let $A_a$ and $A_b$ be $R$-binomial $\cup_1$-dgas with
first cohomology groups finitely generated free $R$-modules.
Suppose $\mcm_{n-1}$ is the $(n-1)$th step in the $1$-minimal
model for both $A_a$ and $A_b$, and that
$\ell_b \colon \mcm_{n-1} \to A_b$ is a structural morphism.
Then the $n$th steps in the $1$-minimal models for $A_a$ and $A_b$
are isomorphic if and only if there is a structural morphism
$\ell_a \colon \mcm_{n-1} \to A_a$ and 
an isomorphism of graded $R$-algebras
$g^{\le 2} \colon \im\bigl(H^{\le 2}(\ell_a)\bigr) \to
\im\bigl(H^{\le 2}(\ell_b)\bigr)$
such that the following diagram commutes:
\begin{equation}
\label{dgm:triangle-2}
\begin{gathered}
\begin{tikzcd}[row sep=1.9pc]
   & H^2(\mcm_{n-1}) \\
 \im\bigl(H^2(\ell_a)\bigr)\ar[rr, "g^2", "\cong" ']
        \ar[ur,<-, "H^2(\ell_a)" ]
        & & \im\bigl(H^2(\ell_b) \bigr)
        \ar[ul, <-,"H^2(\ell_b)" ']
\end{tikzcd}
\end{gathered}
\end{equation}
\end{theorem}

\begin{proof}
Write $K_a = \ker H^2(\ell_a)$ and $K_b = \ker H^2(\ell_b)$,
and consider the two exact sequences
\begin{equation}
\label{dgm:bare}
\begin{tikzcd}[column sep=1.65pc]
0 \ar[r] & K_a \ar[r, "i_a"]
    & H^2(\mcm_{n-1}) \ar[rr, "H^2(\ell_a)"]
        \ar[d, "\id"']
    && \im\bigl(H^2(\ell_a)\bigr) \ar[r] & 0 \\
0 \ar[r] & K_b \ar[r, "i_b"]
    & H^2(\mcm_{n-1}) \ar[rr, "H^2(\ell_b)"]
    && \im\bigl(H^2(\ell_b)\bigr) \ar[r] & 0.
\end{tikzcd}
\end{equation}
By Lemma~\ref{lem:classify}, the $n$th step $\mcm_n^a$ in the
$1$-minimal model for $A_a$ is the Hirsch extension of $\mcm_{n-1}$
with $h$-invariant $i_a\colon K_a \hookrightarrow H^2(\mcm_{n-1})$,
and likewise $\mcm_n^b$ is the Hirsch extension with $h$-invariant
$i_b\colon K_b \hookrightarrow H^2(\mcm_{n-1})$.
Thus, two such Hirsch extensions are isomorphic (as extensions of
$\mcm_{n-1}$) if and only if there is an isomorphism
$e\colon K_a \isom K_b$ with $i_b \circ e = i_a$.

\medskip
\noindent$(\Leftarrow)$
Suppose there exists a structural morphism
$\ell_a\colon\mcm_{n-1}\to A_a$ and an isomorphism of graded $R$-algebras
$g^{\le 2}\colon \im\bigl(H^{\le 2}(\ell_a)\bigr)\to
\im\bigl(H^{\le 2}(\ell_b)\bigr)$ making \eqref{dgm:triangle-2} commute.
Since $H^1(\mcm_{n-1})$ maps isomorphically onto both images under
$H^1(\ell_a)$ and $H^1(\ell_b)$, the commutativity of the degree-$1$
part of \eqref{dgm:triangle-2} is automatic.
For degree $2$: the commutativity of \eqref{dgm:triangle-2} means
$H^2(\ell_b) = g^2 \circ H^2(\ell_a)$, so an element of
$H^2(\mcm_{n-1})$ lies in $K_a$ if and only if it lies in $K_b$.
Hence $K_a = K_b$ as submodules of $H^2(\mcm_{n-1})$. Taking
$e = \id_{K_a}$ gives $i_b \circ e = i_a$, so by 
Lemma~\ref{lem:Hirsch-iso}, $\mcm_n^a \cong \mcm_n^b$.

\medskip
\noindent$(\Rightarrow)$
Suppose the $n$th steps $\mcm_n^a$ and $\mcm_n^b$ are isomorphic.
By Lemma~\ref{lem:classify}, since $A_a$ has a $1$-minimal model
sharing $\mcm_{n-1}$ as its $(n-1)$th step, there exists a structural
morphism $\ell_a\colon \mcm_{n-1}\to A_a$.
By Lemma~\ref{lem:classify} again, this $\ell_a$ determines the
$n$th step $\mcm_n^a$ as the Hirsch extension of $\mcm_{n-1}$ with 
$h$-invariant $i_a\colon K_a\hookrightarrow H^2(\mcm_{n-1})$, and 
likewise $\ell_b$ determines $\mcm_n^b$ with $h$-invariant 
$i_b\colon K_b\hookrightarrow H^2(\mcm_{n-1})$.
Since $\mcm_n^a\cong\mcm_n^b$, there is an isomorphism
$e\colon K_a \isom K_b$ with $i_b\circ e = i_a$.
Define $g^2\colon \im H^2(\ell_a)\to \im H^2(\ell_b)$ by
$g^2(H^2(\ell_a)(x)) = H^2(\ell_b)(x)$ for $x\in H^2(\mcm_{n-1})$:
this is well-defined because $\ker H^2(\ell_a)=K_a=K_b=\ker H^2(\ell_b)$,
and it is an isomorphism with $g^2\circ H^2(\ell_a)=H^2(\ell_b)$,
so \eqref{dgm:triangle-2} commutes.
Setting $g^1 = \id$ on $H^1$ (both images coincide with
$H^1(\mcm_{n-1})$ via isomorphisms induced by $\ell_a$ and $\ell_b$),
the map $g^{\le 2}=(g^1,g^2)$ is an isomorphism of graded $R$-algebras
from $\im\bigl(H^{\le 2}(\ell_a)\bigr)$ to $\im\bigl(H^{\le 2}(\ell_b)\bigr)$.
\end{proof}

We now translate Theorem~\ref{thm:triangle} into the topological setting.

\begin{theorem}
\label{thm:trianglespaces}
With notation and assumptions as above, let $N$ be a nilpotent
group with $\Gamma_n^R(N)=1$ (equivalently, $N=\GGR{N}{n}$), and
suppose both $\GGR{G_a}{n}$ and $\GGR{G_b}{n}$ are isomorphic
to $N$.
Let $q_b \colon X_b \to K(N,1)$ be a classifying map. Then 
$\GGR{G_a}{n+1}$ and $\GGR{G_b}{n+1}$ are isomorphic
if and only if there is a classifying map 
$q_a \colon X_a \to K(N,1)$ and 
an isomorphism of graded $R$-algebras
$g^{\le 2} \colon \im\bigl(H^{\le2}(q_a^\#)\bigr) \isom 
\im\bigl(H^{\le 2}(q_b^\#)\bigr)$
such that the following diagram commutes.
\begin{equation}
\label{dgm:trianglespaces}
\begin{gathered}
\begin{tikzcd}[row sep=1.9pc]
   & H^2(K(N,1);R) \\
 		\im\bigl(H^2(q_a^\#)\bigr)    
 		\ar[rr, "g^2", "\cong" ']
 		\ar[ur,<-, "H^2(q_a^\#)" ]
 		& & \im\bigl( H^2(q_b^\#)\bigr)   
		\ar[ul, <-,"H^2(q_b^\#)" ']
\end{tikzcd}
\end{gathered}
\end{equation}
\end{theorem}

\begin{proof}
Let $\mcm(N)$ be the $1$-minimal model for $C^\ast(K(N,1);R)$.
From Theorem \ref{thm:mnequalsquotient} it follows that 
$\mcm(N) = \mcm_{n-1}(C^\ast(X_a;R))= \mcm_{n-1}(C^\ast(X_b;R))$.
Since $q_b$ is a classifying map, Lemma~\ref{lem:correspondence}
gives that $\ell_b := q_b^\#\circ \psi$ is a structural morphism
$\mcm(N)\to C^*(X_b;R)$.

\medskip
\noindent$(\Rightarrow)$
Suppose $\GGR{G_a}{n+1} \cong \GGR{G_b}{n+1}$. 
Then, by Theorem \ref{thm:mnequalsquotient}, the $n$th steps in the 
$1$-minimal models for $C^\ast(X_a;R)$ and $C^\ast(X_b;R)$
are isomorphic. 
Thus, by Theorem \ref{thm:triangle} there is a structural
morphism $\ell_a \colon \mcm(N) \to C^\ast(X_a;R)$ and an isomorphism
of graded $R$-algebras $g^{\le 2} \colon \im\bigl( H^{\le 2}(\ell_a)\bigr) \to
\im\bigl( H^{\le 2}(q_b^\# \circ \psi)\bigr)$ such that 
\begin{equation}
\label{eq:g2H2}
 g^2 \circ H^2(\ell_a) = H^2(q_b^\# \circ \psi).
 \end{equation}
By Theorem \ref{thm:strict-representability} and 
Lemma \ref{lem:correspondence}, it then follows
that there is a classifying map $r(\ell_a) \colon X_a \to K(N,1)$
with $|r(\ell_a)|^\# \circ \psi = \ell_a$. 
From equation \eqref{eq:g2H2} it then follows that the diagram
\eqref{dgm:trianglespaces} commutes with $q_a = r(\ell_a)$.

\medskip
\noindent$(\Leftarrow)$
Suppose there is a classifying map $q_a \colon X_a \to K(N,1)$
and an isomorphism of graded $R$-algebras $g^{\le 2}$ as in the
statement, making diagram~\eqref{dgm:trianglespaces} commute.
Then since $\psi \colon \mcm(N) \to K(N,1)$ induces an isomorphism
on cohomology, it follows that the diagram
\begin{equation}
\label{dgm:trianglebinomials}
\begin{gathered}
\begin{tikzcd}[row sep=1.9pc]
   & H^2(\mcm(N)) \\
 		\im\bigl(H^2(q_a^\# \circ \psi)\bigr)    
 		\ar[rr, "g^2", "\cong" ']
 		\ar[ur,<-, "H^2(q_a^\#\circ \psi)" ]
 		& & \im\bigl( H^2(q_b^\#\circ \psi)\bigr)   
		\ar[ul, <-,"H^2(q_b^\#\circ \psi)" ']
\end{tikzcd}
\end{gathered}
\end{equation}
commutes, and since 
$\mcm(N) = \mcm_{n-1}(C^\ast(X_a;R))=\mcm_{n-1}(C^\ast(X_b;R))$
it follows from Theorem \ref{thm:triangle} that the $n$th steps
in the $1$-minimal models for $C^\ast(X_a;R)$ and $C^\ast(X_b;R)$
are isomorphic.
From Theorem \ref{thm:mnequalsquotient} it follows that 
$\GGR{G_a}{n+1} \isom \GGR{G_b}{n+1}$ and the 
proof is complete. 
\end{proof}

\begin{remark}
\label{rem:triangle-spaces}
Theorem~\ref{thm:trianglespaces} applies to all pairs of connected
spaces $X_a$ and $X_b$ with finitely generated first cohomology groups
with coefficients in $R$. In particular, it contains
\cite[Thm.~6.5]{Porter-Suciu-2020} as the special case in which the
classifying maps
$f_a \colon X_a \to K(\GGR{G_a}{2},1)$ and
$f_b \colon X_b \to K(\GGR{G_b}{2},1)$ induce epimorphisms
$H^2(f_a)$ and $H^2(f_b)$.
Without that surjectivity hypothesis,
\cite[Thm.~6.5]{Porter-Suciu-2020} does not apply, whereas
Theorem~\ref{thm:trianglespaces} still does. A natural class for which
the surjectivity fails is provided by complements of links in $S^3$
with all linking numbers zero, since the cup product on $H^1$
vanishes there and the classifying map's image in $H^2$ collapses to
zero. The generalized Borromean link complements analyzed in
\S\ref{subsec:borromean} are concrete instances of this phenomenon; a
related $2$-step-equivalence distinction governed by the cokernel
invariant $\kappa_2$ was previously obtained in
\cite[\S12, Prop.~12.5]{Porter-Suciu-2023}.
\end{remark}

\begin{remark}
Note that the isomorphisms $g^{\le 2}$ in
Theorems~\ref{thm:triangle} and~\ref{thm:trianglespaces} are required
only on a subalgebra of the second cohomology, not on all of it.
The family in \S\ref{subsec:borromean} shows why this matters: there
the hypotheses of these theorems hold---so the third steps of the
$1$-minimal models are isomorphic---yet the isomorphism does not
extend to all of the second cohomology. Thus, were the condition
strengthened to require that $g^{\le 2}$ extend to all of the second
cohomology, the conclusion could fail.
\end{remark}


\subsection{The cokernel as a torsion invariant}
\label{subsec:coker-torsion}

The two halves of the structural morphism in degree~$2$ carry
complementary information about the $(n{+}1)$st graded piece of the
lower central series of $G=\pi_1(X)$: its \emph{kernel} records the
torsion-free part, identified by 
Theorem~\ref{thm:mnequalsquotient}, while the
torsion in its \emph{cokernel} records the torsion of $\gr_{n+1}(G)$. 
To make this precise we recall
from \cite[\S12]{Porter-Suciu-2023} the isomorphism invariant
\[
\kappa_n(X)=\Tors\bigl(\coker H^2(\rho_n)\bigr),
\] 
a finite abelian group that is independent of the chosen $1$-minimal 
model and depends only on the isomorphism type of $\pi_1(X)$.

\begin{theorem}
\label{thm:coker-torsion}
Let $X$ be a connected $\Delta$-complex with $H^1(X;\Z)$ and $H^2(X;\Z)$
finitely generated, let $G=\pi_1(X)$, and let
$\rho_n\colon \mcm_n\to C^\ast(X;\Z)$ be a structural morphism for the
integral $1$-minimal model. Suppose $\gr_i(G)$ is torsion-free for
$1\le i\le n$, and -- for parts~\ref{itm:ct-coker} and~\ref{itm:ct-split} --
that $H_2(N_n;\Z)$ is torsion-free, where $N_n=G/\Gamma_{n+1}(G)$. Then:
\begin{enumerate}[label=\textup{(\alph*)}, itemsep=2pt]
\item\label{itm:ct-ker}
$\ker H^2(\rho_n)\cong \gr_{n+1}^0(G)$, a free abelian group whose rank
is that of $\gr_{n+1}(G)$;
\item\label{itm:ct-coker}
$\kappa_n(X)=\Tors\bigl(\coker H^2(\rho_n)\bigr)\cong
   \Tors\bigl(\gr_{n+1}(G)\bigr)$;
\item\label{itm:ct-split}
$\gr_{n+1}(G)\cong \gr_{n+1}^0(G)\oplus \kappa_n(X)$.
\end{enumerate}
\end{theorem}

\begin{proof}
By the hypothesis and Remark~\ref{rmk:grequal}, $\Gamma_j(G)=\Gamma_j^0(G)$
for $1\le j\le n+1$, so $\gr_i(G)=\gr_i^0(G)$ for $i\le n$ and the two
series first differ at level $n+2$.

\ref{itm:ct-ker} By Lemma~\ref{lem:classify}\ref{cl:iii}, $\mcm_{n+1}$ is the
Hirsch extension of $\mcm_n$ with $h$-invariant
$\ker H^2(\rho_n)\hookrightarrow H^2(\mcm_n)$; its adjoined degree-$1$
generators form a basis of $\ker H^2(\rho_n)$, which is therefore free.
By Theorem~\ref{thm:mnequalsquotient}
these generators are identified with a basis of
$\Gamma_{n+1}^0(G)/\Gamma_{n+2}^0(G)=\gr_{n+1}^0(G)$, proving~\ref{itm:ct-ker}.

\ref{itm:ct-coker} 
As noted in \cite[(2.7)]{Porter-Suciu-2020} it follows from
\cite{Stallings} (see also \cite{Dwyer-1975}) that there is an exact sequence
\begin{equation}
\label{ex:Stallings}
\begin{tikzcd}[column sep=32pt]
H_2(X;\Z) \ar[r, "(h_{n+1})_\ast"]
		& H_2(G/\Gamma_{n+1}(G);\Z) \ar[r, "\chi_{n+1}"]
		& \gr_{n+1}(G) \ar[r]
		& 0
\end{tikzcd}
\end{equation}
where $h_{n+1}$ denotes a classifying map
from $X$ to $K(G/\Gamma_{n+1}(G);1)$ in the Postnikov tower
for the lower central series quotients of $\pi_1(X)$.

Since transposition preserves the
torsion of a cokernel,
\begin{equation}
\label{eq:ct-transpose}
   \Tors\bigl(\coker H^2(h_{n+1})\bigr) \cong 
   \Tors\bigl(\coker(\,\overline{H_2(h_{n+1})}\colon
   \overline{H_2(X)}\longrightarrow 
   \overline{H_2(G/\Gamma_{n+1}(G))}\,)\bigr),
\end{equation}
where overlines denote free quotients.

From the assumption that $H_2(G/\Gamma_{n+1};\Z)$
is torsion free, it follows that
\begin{equation}
\label{eq:ct-transpose-tf}
\begin{split}
   \Tors\bigl(\coker(\,\overline{H_2(h_{n+1})}\colon &
   \overline{H_2(X)}\longrightarrow 
   \overline{H_2(G/\Gamma_{n+1}(G))}\,)\bigr)\\
&   = 
   \Tors\bigl(\coker(\,{H_2(h_{n+1})}\colon
   		{H_2(X)}\longrightarrow 
   		{H_2(G/\Gamma_{n+1}(G))}\,)\bigr)
\end{split}   		
\end{equation}
and hence
\begin{equation}
\label{eq:coker}
\Tors \bigl( \coker H^2(h_{n+1}) \bigr)
=  \Tors\bigl(\coker(\,{H_2(h_{n+1})}\colon
   		{H_2(X)}\longrightarrow 
   		{H_2(G/\Gamma_{n+1}(G))}\,)\bigr)
\end{equation}
From the exact sequence in \eqref{ex:Stallings},
\[
\coker(\,{H_2(h_{n+1})}\colon
   		{H_2(X)}\longrightarrow 
   		{H_2(G/\Gamma_{n+1}(G))}\,)
   		= \gr_{n+1}(G)
\]		
and it follows that
\begin{equation}
\label{eq:tgr}
\Tors(\coker H^2(h_{n+1})) = \Tors(\gr_{n+1}(G))
\end{equation}
From the assumption that $\gr_i(G)$ is torsion-free for
$1 \le i \le n$ it follows, as noted in Remark \ref{rmk:grequal},
 that
$\gr_i(G) = \gr_i^0(G)$ for $1 \le i \le n$. Hence,
$\Gamma_j(G) = \Gamma_j^0(G)$ for $1 \le j \le n+1$, and 
\begin{equation}
\label{eq:qtequal}
G/\Gamma_{n+1} = G/\Gamma_{n+1}^{0}(G)
\end{equation}
By equation \eqref{eq:qtequal}, the map $h_{n+1}$
is a classifying map in the Postnikov tower of the quotients
of $G$ by the subgroups $\Gamma_i^0(G)$.
Since $G(\mcm_n) = G/\Gamma_{n+1}^0(G)$, we can assume
by Theorem \ref{thm:strict-representability} and 
Lemma \ref{lem:correspondence} that
the map $h_{n+1} \colon X \to K(G/\Gamma_{n+1}^0(G),1)$
corresponds to $\rho_n$. In particular,
the following diagram commutes.
\[
\begin{tikzcd} [row sep=24pt, column sep=36pt]
		& \mcm_n \ar[dl, "\rho_n" '] \ar[d, "\psi" ]\\
 C^{\ast}(X;\Z)  
		&C^\ast \bigl(  K(G/\Gamma_{n+1}^0(G),1) ;\Z \bigr) 
			\ar[l,  "(h_{n+1})^{\#}" ] 
\end{tikzcd}
\]
By part \eqref{pt:four} of Lemma \ref{lem:Hirsch-central}, 
$\psi$ is a quasi-isomorphism. 
Hence, $H^2(\psi)$ is an isomorphism, and we have that
\[
\coker H^2(h_{n+1}) = \coker H^2(\rho_n)
\]
From equation \eqref{eq:tgr} it follows that
\begin{equation}
\kappa_n(X)= \Tors(\coker H^2(\rho_n)) = \Tors( \gr_{n+1}(G))
\end{equation}
and the proof of \ref{itm:ct-coker} is complete.

\ref{itm:ct-split} follows from~\ref{itm:ct-ker}
and the splitting $\gr_{n+1}(G)=\gr_{n+1}^0(G)\oplus\Tors\gr_{n+1}(G)$ 
of a finitely generated abelian group.
\end{proof}

\begin{remark}
\label{rem:coker-torsion-hyp}
The torsion-freeness of $H_2(N_n;\Z)$ is used in part~\ref{itm:ct-coker}
to justify the passage to free quotients in
\eqref{eq:ct-transpose}. Without it, the same argument identifies
\[
\Tors\bigl(\coker H^2(\rho_n)\bigr) \cong 
\Tors\gr_{n+1}(G)\big/\partial\bigl(\Tors H_2(N_n;\Z)\bigr),
\]
where $\partial\colon H_2(N_n)\to\gr_{n+1}(G)$ is the transgression of the
central extension of $G/\Gamma_{n+1}(G)$ by
$\gr_{n+1}(G)$; equality with $\Tors\gr_{n+1}(G)$
holds precisely when $\partial$ annihilates $\Tors H_2(N_n;\Z)$. The
hypothesis is automatic when $\gr_{\le n}(G)$ is a free Lie algebra: then
$N_n$ is free nilpotent, and in this case
$H_2(N_n;\Z)\cong\gr_{n+1}$ is a free abelian
group. In particular it holds for the generalized Borromean link
complements of \S\ref{subsec:borromean}, where $N_2=F_3/\Gamma_3$ and
$H_2(N_2;\Z)\cong\Z^8$.

Note that if $\gr_i(G)$ is torsion-free for $1 \le i\le n$, then in the exact sequence \eqref{ex:Stallings}
the quotient
$G/\Gamma_{n+1}(G)$ can be replaced by $G/\Gamma_{n+1}^0(G)$.
Thus, $H_2(G/\Gamma_{n+1}^0(G);\Z)$, the map
$(h_{n+1})_\ast$, and hence $\gr_{n+1}(G)$ can be computed
using the $n$th step, $\mcm_n$, 
in the $1$-minimal model for $X$ and a structural morphism from
$\mcm_n$ to $C^\ast(X;\Z)$.
\end{remark}

\begin{remark}
\label{rem:lattice-vs-rational}
Take $R = \Z$. The criterion of Theorem~\ref{thm:trianglespaces} is
integral, and where the kernels $\ker H^2(\rho_n)$ or the algebra
structure on the images differ it can be strictly finer than the
rational data. It is, however, a criterion about the \emph{torsion-free}
graded data (the kernels and images of the structural morphisms) and
is insensitive to the torsion of their cokernels.

The generalized Borromean link complements $Y(k)$ of
\S\ref{subsec:borromean} illustrate this complementarity. There the
kernels $\ker H^2(\rho_2{(k)})$, recording the vanishing pattern of the
triple Massey products, are independent of $k$; in low degrees the
rational $1$-minimal models agree and the torsion-free graded pieces
$\gr_i^0(\pi_1(Y(k)))$ have $k$-independent ranks. The integer Milnor invariants separating the
links is therefore not seen by the kernel/image data of
Theorem~\ref{thm:trianglespaces}, but by the torsion of the cokernel
$\coker H^2(\rho_2{(k)}) = (\Z/k\Z)^2$ --- equivalently, by torsion in the
ordinary lower central series, recorded by the invariant $\kappa_2$ of
\cite[\S12]{Porter-Suciu-2023}.
\end{remark}


\section{Examples}
\label{sec:examples}

We illustrate the results of the previous sections with two examples, 
in each of which the spaces involved have pairwise isomorphic cohomology 
rings but distinct nilpotent quotients. The first 
(\S\ref{subsec:ex-mod2}) is a pair of $2$-complexes with two-generator 
fundamental groups; working over $\F_2$, we distinguish them by 
Theorem~\ref{thm:trianglespaces}. The second (\S\ref{subsec:borromean}) 
is an infinite family of link complements with three-generator 
fundamental groups, indexed by an integer $k \ge 1$; working over $\Z$, 
we distinguish them by torsion in the integral lower central series, 
detected through Theorem~\ref{thm:mnequalsquotient} and the cokernel of 
the structural morphism.

\subsection{A pair of $2$-complexes}
\label{subsec:ex-mod2}

Let $X_0$ be the presentation $2$-complex for the group
\[
G_0 = \langle g_1, g_2 \mid [g_1, g_2^2] \rangle,
\]
and let $X_1$ be the presentation $2$-complex for the group
\[
G_1 = \langle g_1, g_2 \mid [g_1, g_2 g_1 g_2] \rangle.
\]
Both spaces have $H^1(X_i;\Z) = \Z^2$, with basis $\{u_1, u_2\}$
dual to the generators $\{g_1, g_2\}$.
The cup products satisfy $u_1\cup u_2 = 2w$ for a generator
$w\in H^2(X_i;\Z)$, so the cohomology rings $H^\ast(X_0;\Z)$
and $H^\ast(X_1;\Z)$ are isomorphic. Throughout, we write 
$G_i = \pi_1(X_i)$; since $X_i$ is the presentation $2$-complex 
of $G_i$, this is consistent with the notation above.

\medskip\noindent\textit{The $1$-minimal models over $\Z$ and $\F_2$.}
Since both relators lie in the commutator subgroup of the free group
on $g_1, g_2$, we have $\GGR{G_i}{2} = \Z^2$ for $i=0,1$.

\begin{lemma}
\label{lem:mcm1-is-full}
The first step $\mcm_1 = \T_\Z(\{x_1,x_2\})$ with $dx_i=0$ is
the full $1$-minimal model for $C^\ast(X_i;\Z)$ for $i=0,1$.
In particular, $\Gamma_n^0(G_i)/\Gamma_{n+1}^0(G_i)=0$
for $n\ge 2$ and $i=0,1$.
\end{lemma}

\begin{proof}
Let $\rho_1\colon\mcm_1\to C^\ast(X_i;\Z)$ be a structural morphism
with $H^1(\rho_1)(x_j)=u_j$.
Then $u_1\cup u_2([r_i])=2$, so the natural map
$H^2(K(\Z^2,1);\Z)\to H^2(X_i;\Z)$ is a monomorphism.
Hence $\ker H^2(\rho_1)=0$. By Definition~\ref{def:1-min-model}\eqref{min2},
the new generators of $\mcm_2$ are indexed by a basis of 
$\ker H^2(\rho_1)$, so $\bbX_2 = \emptyset$ and $\mcm_2 = \mcm_1$. 
Inductively, $\mcm_n = \mcm_1$ for all $n\ge 1$, and $\mcm_1$ is 
the full $1$-minimal model for $C^\ast(X_i;\Z)$.
The vanishing of the torsion-free LCS quotients for $n\ge 2$
follows immediately.
\end{proof}

Over $\F_2$, the situation is richer.
Since $u_i^2=0$ and $u_1 u_2=0$ mod~$2$ (from either relator),
the map $H^2(\rho_1)$ over $\F_2$ is the zero map, so the second
step in the $1$-minimal model over $\F_2$ is nontrivial.

\begin{lemma}
\label{lem:mcm2-mod2}
The dga
\[
\mcm_2 = \T_{\F_2}(\{x_1, x_2, x_{1,1}, x_{1,2}, x_{2,2}\}),
\]
with $dx_i=0$, $dx_{i,i}=x_i\otimes x_i$ for $i=1,2$, and
$dx_{1,2}=x_1\otimes x_2$, is the second step in the $1$-minimal
model for $C^\ast(X_i;\F_2)$ for both $i=0,1$.
\end{lemma}

\begin{proof}
Since $H^2(\mcm_1;\F_2)$ has basis $\{[x_1^2], [x_2^2], [x_1 x_2]\}$
over $\F_2$, and each of these vanishes in $H^2(X_i;\F_2)$ (as 
$u_j^2 = 0$ and $u_1 u_2 = 2w \equiv 0 \pmod 2$), the kernel 
$\ker H^2(\rho_1)$ over $\F_2$ is the full $H^2(\mcm_1;\F_2)$, 
with basis $\{[x_1^2], [x_2^2], [x_1 x_2]\}$.
The corresponding generators $x_{1,1}, x_{2,2}, x_{1,2}$ with 
differentials as above give the Hirsch extension $\mcm_1\hookrightarrow\mcm_2$,
which is therefore the second step in the $1$-minimal model.
\end{proof}

\medskip\noindent\textit{A basis for $H^2(\mcm_2)$.}
We use the Serre spectral sequence of the central extension
$G(\mcm_1)\hookrightarrow G(\mcm_2)$ with
$E_2^{p,q}= H^p(\mcm_1)\otimes
H^q(\T_{\F_2}(\{x_{1,1},x_{1,2},x_{2,2}\}),d_{\bz})$.

\begin{lemma}
\label{lem:basis-H2}
The $\F_2$-vector space $H^2(\mcm_2)$ has dimension~$4$,
with basis
\[
\bigl\{[\wht{\gamma}_1],\;[\wht{\gamma}_2],\;
[x_1 x_{1,2}+x_{1,1}x_2],\;[x_1 x_{2,2}+x_{1,2}x_2]\bigr\},
\]
where
\begin{align*}
\gamma_1 &= x_1\cup_1 x_{1,1} + x_1\cup_1 x_1\cup_1 x_1,
  \qquad \wht{\gamma}_1 = x_1\gamma_1 + x_{1,1}^2 + \gamma_1 x_1,\\
\gamma_2 &= x_2\cup_1 x_{2,2} + x_2\cup_1 x_2\cup_1 x_2,
  \qquad \wht{\gamma}_2 = x_2\gamma_2 + x_{2,2}^2 + \gamma_2 x_2.
\end{align*}
In particular, $\wht{\gamma}_1$ involves only $x_1$,
and $\wht{\gamma}_2$ involves only $x_2$.
\end{lemma}

\begin{proof}
The $E_2$-cocycles of total degree $2$ not in the image of $d_2$
are represented by $x_{1,1}^2$, $x_{2,2}^2$,
$x_{1,1}x_2+x_{1,2}x_1$, and $x_{2,2}x_1+x_{1,2}x_2$.
We lift each to a cocycle in $\mcm_2$.

The classes $[x_1x_{1,2}+x_{1,1}x_2]$ and
$[x_1x_{2,2}+x_{1,2}x_2]$ are immediate: both are cocycles in $\mcm_2$,
representing elements of the Massey products 
$\langle x_1,x_1,x_2\rangle$ and
$\langle x_1,x_2,x_2\rangle$ respectively.

For $x_{1,1}^2$: using the Hirsch identity and $d(x_1\cup_1 x_1)=0$,
\[
d(x_1\cup_1 x_{1,1}) = x_1x_{1,1}+x_{1,1}x_1
    +d(x_1\cup_1 x_1\cup_1 x_1).
\]
Setting $\gamma_1 = x_1\cup_1 x_{1,1}+x_1\cup_1 x_1\cup_1 x_1$
gives $d\gamma_1 = x_1x_{1,1}+x_{1,1}x_1$, so
$[x_1x_{1,1}+x_{1,1}x_1]=0$ in $H^2(\mcm_2)$ and
$0\in\langle x_1,x_1,x_1\rangle$.
Then $\wht{\gamma}_1 = x_1\gamma_1+x_{1,1}^2+\gamma_1 x_1$
is a cocycle projecting to $x_{1,1}^2$ in $E_2$, and
$\wht{\gamma}_2$ is constructed analogously.

The four classes are linearly independent (they project to
distinct basis elements of $E_\infty$) and span $H^2(\mcm_2)$.
\end{proof}

\medskip\noindent\textit{Distinguishing the mod~$2$ quotients.}
We now apply Theorem~\ref{thm:triangle} to show that the mod~$2$
nilpotent quotients of $G_0$ and $G_1$ are non-isomorphic.

\begin{proposition}
\label{prop:ex1}
With notation as above:
\begin{enumerate}[label=(\roman*), itemsep=2pt, topsep=0pt]
\item\label{it:abel} 
$G_i/\Gamma_{2}^0(G_i)= \Z^2$ for $i=0,1$.
\item\label{it:tf}
    $\Gamma_n^0(G_i)/\Gamma_{n+1}^0(G_i)=0$
    for $n\ge 2$ and $i=0,1$.
\item\label{it:dim0}
    $\dim_{\F_2}\bigl(\Gamma_3^2(G_0)/\Gamma_4^2(G_0)\bigr)=4$.
\item\label{it:dim1}
    $\dim_{\F_2}\bigl(\Gamma_3^2(G_1)/\Gamma_4^2(G_1)\bigr)=3$.
\end{enumerate}
\end{proposition}

\begin{proof}
Parts~\ref{it:abel} and~\ref{it:tf} follow from
Lemma~\ref{lem:mcm1-is-full}.
For parts~\ref{it:dim0} and~\ref{it:dim1}, recall that
$\dim(\Gamma_3^2(G)/\Gamma_4^2(G)) = \dim\ker H^2(\rho_2)$
for any structural morphism $\rho_2\colon\mcm_2\to C^\ast(X;\F_2)$, where $\pi_1(X)=G$. 
Let $\rho^{(i)}\colon\mcm_2\to C^\ast(X_i;\F_2)$ be a structural morphism
with $x_j\mapsto c_j$, $[c_j]=u_j$. We compute $\ker H^2(\rho^{(i)})$ using 
the basis of Lemma~\ref{lem:basis-H2}.

\smallskip
\noindent\textit{Step~1: $[\wht\gamma_1]$ and $[\wht\gamma_2]$ lie 
in $\ker H^2(\rho^{(i)})$ for both $i=0,1$.}
For each $i\in\{0,1\}$ and each $j\in\{1,2\}$, the map 
$p_j\colon X_i \to S^1 = K(\Z,1)$ sending $g_j$ to the generator 
and the other generator to the basepoint is well-defined: in each 
case the relator is a commutator, so its image in $\Z$ is trivial.
By Theorem~\ref{thm:1-min-lift}, $p_j$ lifts to 
$\widehat{p}_j \colon \mcm(S^1)\to\mcm_2(X_i)$ with 
$\rho^{(i)}\circ\widehat{p}_j \simeq p_j^\#\circ\psi_{S^1}$. 
Since $\wht\gamma_j$ involves only $x_j$, it lies in 
$\im(\widehat{p}_j)$, and as $H^2(S^1;\F_2)=0$, the class 
$[\wht\gamma_j]$ lies in $\ker H^2(\rho^{(i)})$.

\smallskip
\noindent\textit{Step~2.}
By \cite[Thm.~2]{Porter}, $\langle u_1,u_1,u_2\rangle=\{0\}$ in
$H^2(X_i;\F_2)$ for both $i$.
Since $H^2(\rho^{(i)})[x_1x_{1,2}+x_{1,1}x_2]
\in\langle u_1,u_1,u_2\rangle$,
the class $[x_1x_{1,2}+x_{1,1}x_2]$ lies in
$\ker H^2(\rho^{(i)})$ for both $i=0,1$.

\smallskip
\noindent\textit{Step~3.}
By \cite[Thm.~2]{Porter}, $\langle u_1,u_2,u_2\rangle=\{0\}$ in
$H^2(X_0;\F_2)$ and $\langle u_1,u_2,u_2\rangle=\{w\}$ in
$H^2(X_1;\F_2)$.
Hence $[x_1x_{2,2}+x_{1,2}x_2]\in\ker H^2(\rho^{(0)})$ but
$H^2(\rho^{(1)})[x_1x_{2,2}+x_{1,2}x_2]=w\ne 0$.

\smallskip
Combining: $\ker H^2(\rho^{(0)})\cong\F_2^4$,
giving part~\ref{it:dim0}; and $\ker H^2(\rho^{(1)})\cong\F_2^3$,
spanned by $[\wht{\gamma}_1]$, $[\wht{\gamma}_2]$, and 
$[x_1x_{1,2}+x_{1,1}x_2]$, giving part~\ref{it:dim1}.
\end{proof}

In particular, $\Gamma_3^2(G_0)/\Gamma_4^2(G_0)
\not\cong \Gamma_3^2(G_1)/\Gamma_4^2(G_1)$,
so $G_0/\Gamma_4^2(G_0)
\not\cong G_1/\Gamma_4^2(G_1)$
by Theorem~\ref{thm:triangle}.

\subsection{Generalized Borromean rings}
\label{subsec:borromean}

In this section we distinguish an infinite family of $3$-component link
complements in $S^3$ which are otherwise hard to tell apart: their
integral cohomology rings agree and their rational $1$-minimal models
agree in low degrees, yet the integer multiplicity of their lowest Milnor
invariants is detected by the integral $1$-minimal model. 
The argument
computes the cokernel of the structural morphism $H^2(\rho_2)$; its
torsion records this multiplicity as torsion in the ordinary lower
central series, invisible to the torsion-free Stallings tower.
This example illustrates the connection between the first nonzero
Milnor invariants of a link and successive quotients in the lower
central series of the link complement.

For each $k\ge 1$, let $\beta_k = [\sigma_1^2, \sigma_2^{2k}]$ be the
pure braid on three strands depicted in Figure~\ref{borromean}, let
$C(k)\subset S^3$ be its closure, and let $Y(k) = S^3 \setminus C(k)$
denote the complement. The case $k=1$ gives the classical Borromean
rings; for general $k\ge 1$, the family $\{C(k)\}$ realizes the integer
values $\bar\mu_{123}(C(k)) = k$ of one of the lowest 
non-vanishing Milnor invariants, with all pairwise linking numbers zero.
Then for each $k$, we have that 
$H^1(Y(k);\Z)=\Z^3$ has basis $u_1, u_2, u_3$ where the $u_i$
correspond by Alexander duality to meridians of the three components
of $C(k)$, and
$H^2(Y(k);\Z)=\Z^2$ generated by Lefschetz duals
$\gamma_{i,j}$ to paths between distinct components.

\begin{lemma}
\label{lem:mcm2link}
The dga
\[
\mcm_2 = \T_{\Z}(x_1, x_2, x_3, x_{12}, x_{13}, x_{23})
\]
with $dx_i=0$ and $dx_{ij}=x_i \ot x_j$, is the second step
in the $1$-minimal model for $C^\ast(Y(k);\Z)$ for all
$k \ge 1$.
\end{lemma}

\begin{proof}
Given the basis above for $H^1(Y(k);\Z)$ it follows that the first
step $\mcm_1$ for the $1$-minimal model is
$\T_{\Z}(x_1, x_2, x_3)$ with $dx_i=0$ with a structural morphism
$\rho{(k)}_1 \colon \mcm_1 \to C^\ast(Y(k);\Z)$ given by a map
that sends each $x_i$ to a cocycle representative of $u_i$.

Since all the linking numbers in $C(k)$ are zero, it follows that
all cup products $[x_i] \cdot [x_j]$ vanish in $H^2(Y(k);\Z)$.
Thus, $\mcm_2$ is obtained from $\mcm_1$ by adding generators
$x_{ij}$ for $i<j$ with $dx_{ij} = x_i \ot x_j$.
\end{proof}

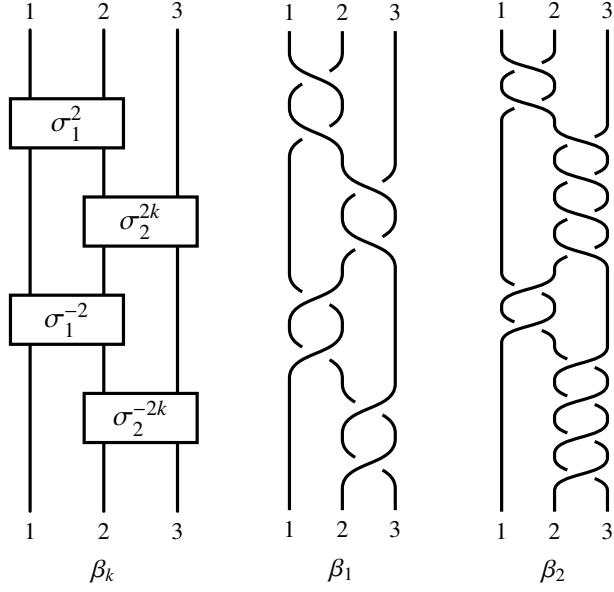
\begin{figure}[ht]
\centering
\begin{tikzpicture}[line width=1.2, line cap=round, scale=0.65, 
baseline=(current bounding box.center)]
\draw (0,-0.4)--(0,9.4);
\draw (1.5,-0.4)--(1.5,9.4);
\draw (3,-0.4)--(3,9.4);
\filldraw[fill=white] (-0.4,7) rectangle (1.9,8);
\node at (0.75,7.5) {$\sigma_1^{2}$};
\filldraw[fill=white] (1.1,5) rectangle (3.4,6);
\node at (2.25,5.5) {$\sigma_2^{2k}$};
\filldraw[fill=white] (-0.4,3) rectangle (1.9,4);
\node at (0.75,3.5) {$\sigma_1^{-2}$};
\filldraw[fill=white] (1.1,1) rectangle (3.4,2);
\node at (2.25,1.5) {$\sigma_2^{-2k}$};
\node[above] at (0,9.4) {\footnotesize $1$};
\node[above] at (1.5,9.4) {\footnotesize $2$};
\node[above] at (3,9.4) {\footnotesize $3$};
\node[below] at (0,-0.4) {\footnotesize $1$};
\node[below] at (1.5,-0.4) {\footnotesize $2$};
\node[below] at (3,-0.4) {\footnotesize $3$};
\node[below=14pt] at (1.5,-0.4) {\small $\beta_k$};
\end{tikzpicture}
\hspace{0.8cm}
\begin{tikzpicture}[baseline=(current bounding box.center)]
\braid[number of strands=3, height=0.73cm, width=0.7cm, gap=0.18,
       style strands={1,2,3}{line width=1.2pt}]
   (b1) s_1 s_1 s_2 s_2 s_1^{-1} s_1^{-1} s_2^{-1} s_2^{-1};
\foreach \i in {1,2,3} {
  \node[above] at (b1-\i-s) {\footnotesize $\i$};
  \node[below] at (b1-\i-e) {\footnotesize $\i$};
}
\node[below=14pt] at (b1-2-e) {\small $\beta_1$};
\end{tikzpicture}
\hspace{0.8cm}
\begin{tikzpicture}[baseline=(current bounding box.center)]
\braid[number of strands=3, height=0.49cm, width=0.7cm, gap=0.18,
       style strands={1,2,3}{line width=1.2pt}]
   (b2) s_1 s_1 s_2 s_2 s_2 s_2 s_1^{-1} s_1^{-1} s_2^{-1} s_2^{-1} s_2^{-1} s_2^{-1};
\foreach \i in {1,2,3} {
  \node[above] at (b2-\i-s) {\footnotesize $\i$};
  \node[below] at (b2-\i-e) {\footnotesize $\i$};
}
\node[below=14pt] at (b2-2-e) {\small $\beta_2$};
\end{tikzpicture}
\caption{The pure braid $\beta_k = [\sigma_1^2,\sigma_2^{2k}]$ on three
strands, shown schematically on the left and drawn explicitly for
$k=1, 2$ on the right. The link $C(k) \subset S^3$ is the closure
of $\beta_k$; for $k=1$ this recovers the classical Borromean rings.}
\label{borromean}
\end{figure}

Since all products of elements in $H^1(\mcm_2)$ are zero, it follows
that all Massey triple products in $\mcm_2$ are defined and have
zero indeterminacy.

\begin{lemma}
\label{lem:H2Mtp}
$H^2(\mcm_2)= \Z^{8}$ with basis given by the following Massey triple
products:
$ \la [x_{1}], [x_{1}], [x_{2}]\ra$,
$ \la [x_{1}], [x_{2}], [x_{2}]\ra$,
$ \la [x_{1}], [x_{1}], [x_{3}]\ra$,
$ \la [x_{1}], [x_{3}], [x_{3}]\ra$,
$ \la [x_{2}], [x_{2}], [x_{3}]\ra$,
$ \la [x_{2}], [x_{3}], [x_{3}]\ra$,
$ \la [x_{1}], [x_{2}], [x_{3}]\ra$, and
$ \la [x_{1}], [x_{3}], [x_{2}]\ra$.
\end{lemma}

\begin{proof}
The result follows using the spectral sequence of the Hirsch
extension from $\mcm_1$ to $\mcm_2$.
A direct computation shows that the
 $E_{\infty}^{p,q}$ terms with $p+q=2$ are
$E_{\infty}^{1,1} = E_3^{1,1}$ and that representative
cocycles for the Massey triple products above project to a basis for
$E_{\infty}^{1,1}$.
\end{proof}

The $k$-dependence of the
link is carried by the two Massey products, $\la[x_1],[x_2],[x_3]\ra$
and $\la[x_1],[x_3],[x_2]\ra$, through the structural morphism
$\rho_2{(k)} \colon \mcm_2 \to C^\ast(Y(k);\Z)$, as recorded in the next
lemma.

\begin{lemma}
\label{lem:rho2image}
Let $\rho_2{(k)} \colon \mcm_2 \to C^\ast(Y(k);\Z)$ be a structural
morphism. With respect to the Massey-product basis of
$H^2(\mcm_2) = \Z^8$ from Lemma~\ref{lem:H2Mtp}, the induced map
$H^2(\rho_2{(k)})$ sends the six classes with a repeated index to~$0$,
and sends the two remaining classes by
\[
   \la [x_1], [x_2], [x_3] \ra \longmapsto -k\,\gamma_{1,3},
   \qquad
   \la [x_1], [x_3], [x_2] \ra \longmapsto k\,\gamma_{1,2},
\]
where $\{\gamma_{1,2}, \gamma_{1,3}\}$ is a basis for
$H^2(Y(k);\Z) = \Z^2$. Consequently:
\begin{enumerate}[label=\textup{(\alph*)}, itemsep=2pt, topsep=0pt]
\item\label{itm:rho2-ker}
$\ker H^2(\rho_2{(k)})$ is the rank-$6$ free direct summand of
$H^2(\mcm_2)$ spanned by the six repeated-index classes, independently
of~$k$.
\item\label{itm:rho2-im}
$\im H^2(\rho_2{(k)}) = \bigl\langle\, k\,\gamma_{1,2},\, k\,
\gamma_{1,3}\,\bigr\rangle = k\cdot H^2(Y(k);\Z)$.
\item\label{itm:rho2-coker}
$\coker H^2(\rho_2{(k)}) = H^2(Y(k);\Z)\big/\im H^2(\rho_2{(k)})
   \cong (\Z/k\Z)^2$.
\end{enumerate}
\end{lemma}

\begin{proof}
Since all products of elements of $H^1(\mcm_2)$ vanish, every triple
Massey product of the generators $[x_1],[x_2],[x_3]$ is defined with
zero indeterminacy, and $\rho_2{(k)}$ induces an isomorphism on $H^1$
carrying $[x_i]$ to the meridian class $u_i \in H^1(Y(k);\Z)$. By
naturality of Massey products, $H^2(\rho_2{(k)})$ sends each basis
class $\la [x_i],[x_j],[x_k]\ra$ to the corresponding product
$\la u_i, u_j, u_k\ra \in H^2(Y(k);\Z)$. The six products with a
repeated index vanish on $Y(k)$ (the link obtained by removing 
anyone of the components in $C(k)$ is the unlink), giving the 
six zero values; the two products with distinct
indices are computed in \cite[Ex.~3, p.~46]{Porter}:
$\la u_1, u_2, u_3\ra = -k\,\gamma_{1,3}$ and
$\la u_1, u_3, u_2\ra = k\,\gamma_{1,2}$. 
This establishes the displayed
values, and statement~\ref{itm:rho2-ker} follows since the six
repeated-index basis classes span a rank-$6$ free direct summand of
$H^2(\mcm_2)=\Z^8$. For~\ref{itm:rho2-im}, the pair
$\{-\gamma_{1,3},\,\gamma_{1,2}\}$ is a basis for $H^2(Y(k);\Z)=\Z^2$,
so the image lattice is exactly $k\cdot\Z^2$. 
Statement~\ref{itm:rho2-coker}
is then immediate: $\Z^2/k\Z^2 \cong (\Z/k\Z)^2$.
\end{proof}

\begin{proposition}
\label{prop:borromean}
Let $Y(k) = S^3\setminus C(k)$ be the complement of the link in 
Figure~\ref{borromean}. Then:
\begin{enumerate}
\item 
\label{itm:1}
$\pi_1(Y(k))/\Gamma_2^0\bigl(\pi_1(Y(k))\bigr) 
\cong \Z^3$ for all $k\ge 1$.
\item 
\label{itm:2}
$\Gamma_2^0\bigl(\pi_1(Y(k))\bigr)/\Gamma_3^0\bigl(\pi_1(Y(k))\bigr) \cong \Z^3$ 
for all $k\ge 1$.
\item 
\label{itm:3}
$\Gamma_3^0\bigl(\pi_1(Y(k))\bigr)/\Gamma_4^0
\bigl(\pi_1(Y(k))\bigr) \cong \Z^6$ 
for all $k\ge 1$. Moreover, for all $k,\ell\ge 1$,
\[
\pi_1(Y(k))/\Gamma_4^0\bigl(\pi_1(Y(k))\bigr) \cong 
\pi_1(Y(\ell))/\Gamma_4^0\bigl(\pi_1(Y(\ell))\bigr).
\]
\item 
\label{itm:4}
For $|k|\ne |\ell|$, the fundamental groups 
$\pi_1(Y(k))$ and $\pi_1(Y(\ell))$
are not isomorphic. 
More precisely, in the \emph{ordinary} lower central
series,
\[
\gr_3\bigl(\pi_1(Y(k))\bigr) \cong \Z^6 \oplus (\Z/k\Z)^2,
\]
so $\pi_1(Y(k))/\Gamma_4\bigl(\pi_1(Y(k))\bigr) \not\cong
\pi_1(Y(\ell))/\Gamma_4\bigl(\pi_1(Y(\ell))\bigr)$ 
already at the fourth stage,
whereas the torsion-free quotients 
$\pi_1(Y(k))/\Gamma_s^0\bigl(\pi_1(Y(k))\bigr)$ have graded
ranks independent of~$k$ for $2 \le s \le 4$.
\end{enumerate}
\end{proposition}

\begin{proof}
Part \eqref{itm:1} follows from 
$\pi_1(Y(k))/\Gamma_2^0\bigl(\pi_1(Y(k))\bigr) 
\cong H_1(Y(k);\Z) \cong H^1(Y(k);\Z) = \Z^3$.

Part \eqref{itm:2} follows since 
$\Gamma_2^0\bigl(\pi_1(Y(k))\bigr)/\Gamma_3^0
\bigl(\pi_1(Y(k))\bigr) $ is, 
by Theorem~\ref{thm:mnequalsquotient}, the abelian group 
$\ker H^2(\rho_1{(k)})$, which 
by the proof of Lemma~\ref{lem:mcm2link} equals~$\Z^3$.

For Part~\eqref{itm:3}, the rank claim follows in the same way: 
by Theorem~\ref{thm:mnequalsquotient}, 
$\Gamma_3^0/\Gamma_4^0 \cong \ker H^2(\rho_2{(k)})$,
which by Lemma~\ref{lem:rho2image}\ref{itm:rho2-ker} is free of rank~$6$.

For the isomorphism statement in Part~\eqref{itm:3}, observe that
by Lemma~\ref{lem:mcm2link} the second step $\mcm_2$ in the
$1$-minimal model for $C^\ast(Y(k);\Z)$ is the same dga for every~$k$.
The kernel
\[
K = \ker\bigl(H^2(\rho_2{(k)})\colon H^2(\mcm_2)
\longrightarrow H^2(Y(k);\Z)\bigr)
\]
is also $k$-independent: by Lemma~\ref{lem:H2Mtp} the eight basis
classes of $H^2(\mcm_2)$ are Massey triple products of 
$[x_1],[x_2],[x_3]$,
and by Lemma~\ref{lem:rho2image}\ref{itm:rho2-ker} exactly
the same six of these map to zero in $H^2(Y(k);\Z)$ for every~$k$;
namely, all those with a repeated index.  Only
$\la [x_1],[x_2],[x_3]\ra$ and $\la [x_1],[x_3],[x_2]\ra$ 
are mapped to nonzero elements in $H^2(Y(k);\Z)$
with values that scale by~$k$ but whose vanishing pattern is
$k$-independent.  By Lemma~\ref{lem:classify}\ref{cl:iii},
the third step $\mcm_3$ is the Hirsch extension of $\mcm_2$ with
invariant given by the inclusion of $K$ in
$H^2(\mcm_2)$.
Hence $\mcm_3$ is also common to all~$Y(k)$.
Theorem~\ref{thm:mnequalsquotient} then gives
$\pi_1(Y(k))/\Gamma_4^0 \cong G(\mcm_3) \cong 
\pi_1(Y(\ell))/\Gamma_4^0$.

For Part~\eqref{itm:4}, recall from \cite{Porter-Suciu-2023} the
isomorphism invariant 
$\kappa_2(Y(k)) = \Tors\bigl(\coker H^2(\rho_2{(k)})\bigr)$,
which depends only on $\pi_1(Y(k))$. By
Lemma~\ref{lem:rho2image}\ref{itm:rho2-coker}, 
$\coker H^2(\rho_2{(k)})\cong (\Z/k\Z)^2$ is finite, 
so $\kappa_2(Y(k)) = (\Z/k\Z)^2$. For $|k|\ne|\ell|$
we have $(\Z/k\Z)^2 \not\cong (\Z/\ell \Z)^2$, hence
$\kappa_2(Y(k)) \not\cong \kappa_2(Y(\ell))$, and therefore
$\pi_1(Y(k)) \not\cong \pi_1(Y(\ell))$ by the 
$\pi_1$-invariance of $\kappa_2$
\cite{Porter-Suciu-2023}.

To locate the distinction within the lower central series, we assemble
the third graded piece from its free and torsion parts. The hypotheses
of Theorem~\ref{thm:coker-torsion} hold with $n=2$: by
Parts~\eqref{itm:1}--\eqref{itm:2}, $\gr_1$ and $\gr_2$ are torsion-free,
and $N_2=\pi_1(Y(k))/\Gamma_3=F_3/\Gamma_3$ is free $2$-step 
nilpotent, so $H_2(N_2;\Z)\cong\Z^8$ is torsion-free 
(Remark~\ref{rem:coker-torsion-hyp}).
Hence $\gr_3^0(\pi_1(Y(k))) = \ker H^2(\rho_2{(k)}) \cong \Z^6$
(Theorem~\ref{thm:coker-torsion}\ref{itm:ct-ker}, recovering
Part~\eqref{itm:3}), while
$\Tors\gr_3(\pi_1(Y(k))) \cong \kappa_2(Y(k)) = (\Z/k\Z)^2$
(Theorem~\ref{thm:coker-torsion}\ref{itm:ct-coker}). A direct computation
from the Milnor presentation of $\pi_1(Y(k))$, or equivalently, from the
Magnus expansion of the relators $\beta_k(x_i)\,x_i^{-1}$, confirms
\[
   \gr_3\bigl(\pi_1(Y(k))\bigr) \cong \Z^6 \oplus (\Z/k\Z)^2.
\]
Since $\gr_3$ is a subquotient of $\pi_1(Y(k))/\Gamma_4$, distinct values
of $|k|$ yield non-isomorphic fourth nilpotent quotients. By contrast, the
torsion-free graded ranks are independent of~$k$ for $2\le s\le 4$: over
the rationals the complements $Y(k)$ are $2$-step equivalent
\cite{Porter-Suciu-2023}, and the ranks $3,3,6$ of
$\gr_1,\gr_2,\gr_3$ found above do not depend on~$k$.
(Whether the $Y(k)$ share a common rational homotopy type, which would
yield rank-independence in all degrees, is not known
\cite{Porter-Suciu-2023}.)
Up to the level $K(G/\Gamma_4^0(G),1)$, the torsion-free Stallings tower does 
not distinguish the links; the distinction is carried entirely by the torsion of the 
ordinary graded.
\end{proof}

\begin{remark}
\label{rem:borromean-mechanism}
By Theorem~\ref{thm:coker-torsion}, the torsion summand 
$(\Z/k\Z)^2$ in
$\gr_3(\pi_1(Y(k)))$ is precisely
$\kappa_2(Y(k))=\Tors\coker H^2(\rho_2{(k)})$, the group-theoretic
footprint of the Milnor invariant $\bar\mu_{123}(C(k))=k$ in the
associated graded Lie ring. This finite abelian group is invisible 
to the rational $1$-minimal model and to the torsion-free Stallings 
tower, yet is recorded integrally by the cokernel of the structural 
morphism in degree~$2$.
\end{remark}

The arguments in Propositions~\ref{prop:ex1} and~\ref{prop:borromean} 
rely on the relevant Massey products having zero indeterminacy. In the companion 
paper~\cite{Porter-Suciu-GMP} we introduce generalized Massey products, 
which have smaller indeterminacy and yield correspondingly stronger 
invariants. There we exhibit spaces that these invariants distinguish 
but the usual Massey products do not, including spaces all of whose 
ordinary Massey products contain zero yet which carry a nonzero 
generalized Massey product, thereby detecting nonformality invisible 
to the classical theory.

\appendix
\section{Background material}
\label{app:background}

This appendix collects, for the reader's convenience, two kinds of 
background material used in the proofs above.
The first part (\S\ref{subsec:dga-homotopy}) recalls the notion of 
homotopy between binomial $\cup_1$-dga maps
(\cite{Porter-Suciu-2023}) and three results invoked in the proofs of 
Lemmas~\ref{lem:classify} and~\ref{lem:Hirsch-iso} and 
Theorems~\ref{thm:group-well-defined} and~\ref{thm:mnequalsquotient}.
The second part (\S\ref{subsec:delta}) fixes notation for $\Delta$-sets 
and their cochain algebras, following 
Rourke--Sanderson~\cite{Rourke-Sanderson}, 
Hatcher~\cite{Hatcher}, and Friedman~\cite{Friedman}.

\subsection{Homotopic maps and homotopy liftings}
\label{subsec:dga-homotopy}
Let $I = [0,1]$ viewed as a simplicial complex, and let
$C = C^*(I;R)$ be its cochain algebra over $R$.
Then $C^0 \cong R \oplus R$ with generators $t_0, t_1$ corresponding
to the endpoints $0$ and $1$, and $C^1 \cong R$ with generator $u$.
The differential $d\colon C^0 \to C^1$ is given by
$dt_0 = -u$ and $dt_1 = u$, the multiplication by
$t_i t_j = \delta_{ij} t_i$, $t_0 u = u t_1 = u$,
and $u t_0 = t_1 u = 0$; the unit is $t_0 + t_1$ and
$H^*(C) = R$.
Let $\eta_i \colon C^*(I;R) \to R$ be evaluation at the $i$-th
endpoint: $\eta_i(t_j) = \delta_{ij}$, $\eta_i(u) = 0$.

\begin{definition}[{\cite{Porter-Suciu-2023}}]
\label{def:homotopy}
Two dga maps $\varphi_0, \varphi_1 \colon A \to B$ are
\emph{homotopic} (written $\varphi_0 \simeq \varphi_1$) if
there exists a dga map $\Phi \colon A \to B \otimes_R C^*(I;R)$
such that $(\id_B \otimes \eta_i) \circ \Phi = \varphi_i$
for $i = 0, 1$.
\end{definition}

The following two lemmas are used in the proof of
Lemma~\ref{lem:classify}.

\begin{lemma}[{\cite[Lem. 7.12]{Porter-Suciu-2023}}]
\label{lem:H1-homotopy}
Let $(A, d_A)$ be an $R$-binomial $\cup_1$-dga over $R = \Z$ or $\F_p$
with $H^0(A) = R$ and $H^1(A)$ a finitely generated free $R$-module.
If $\varphi_0, \varphi_1 \colon (\T_R(\bbX), d_{\bz}) \to (A, d_A)$
are morphisms of $R$-binomial $\cup_1$-dgas with
$H^1(\varphi_0) = H^1(\varphi_1)$,
then $\varphi_0 \simeq \varphi_1$.
\end{lemma}

\begin{lemma}[Homotopy Lifting Lemma, {\cite[Lem.~10.2]{Porter-Suciu-2023}}]
\label{lem:homotopy-lift}
Let $(A, d_A)$ and $(A', d_{A'})$ be $R$-binomial $\cup_1$-dgas
over $R = \Z$ or $\F_p$ with $H^0 = R$ and $H^1$
a finitely generated free $R$-module.
Let $\rho \colon \mcm(A) \to A$ and
$\rho' \colon \mcm(A') \to A'$ be $1$-minimal models,
with Hirsch-extension inclusions
$j_n\colon\mcm_n\inj\mcm_{n+1}$ and $j_n'\colon\mcm_n'\inj\mcm_{n+1}'$,
and let $\varphi \colon A \to A'$ be a morphism.
Suppose for some $n \ge 1$ there exist a morphism
$f_n \colon \mcm_n \to \mcm_n'$ and a homotopy
$\Phi_n \colon \mcm_n \to A' \otimes_R C^*(I;R)$
between $\varphi \circ \rho_n$ and $\rho_n' \circ f_n$.
Then:
\begin{enumerate}[label=(\roman*), itemsep=2pt]
\item \label{hl1}
There is a unique morphism
$f_{n+1} \colon \mcm_{n+1} \to \mcm_{n+1}'$
with $f_{n+1} \circ j_n = j_n' \circ f_n$, and a homotopy
$\Phi_{n+1} \colon \mcm_{n+1} \to A' \otimes_R C^*(I;R)$
between $\varphi \circ \rho_{n+1}$ and $\rho_{n+1}' \circ f_{n+1}$
extending $\Phi_n$.
\item \label{hl2}
If $f_n$ is an isomorphism and $H^2(\varphi)$ is a monomorphism,
then $f_{n+1}$ is also an isomorphism.
\end{enumerate}
\end{lemma}

The following theorem is used in the proofs of
Theorems~\ref{thm:group-well-defined} and~\ref{thm:mnequalsquotient}.

\begin{theorem}[Homotopy Lifting Theorem,
  {\cite[Thm.~10.3]{Porter-Suciu-2023}}]
\label{thm:1-min-lift}
Let $(A, d_A)$ and $(A', d_{A'})$ be $R$-binomial $\cup_1$-dgas 
with $H^0 = R$ and $H^1$ a finitely generated free $R$-module,
and let $\rho \colon \mcm(A) \to A$,
$\rho' \colon \mcm(A') \to A'$ be $1$-minimal models. Then:
\begin{enumerate}[label=(\roman*), itemsep=2pt]
\item For any morphism $\varphi \colon A \to A'$ there exists 
a morphism $\widehat\varphi \colon \mcm(A) \to \mcm(A')$, unique 
up to homotopy, such that the square
\[
\begin{tikzcd}
\mcm(A) \ar[r, dashed, "\widehat\varphi"] \ar[d, "\rho"'] &
\mcm(A') \ar[d, "\rho'"] \\
A \ar[r, "\varphi"'] & A'
\end{tikzcd}
\]
commutes up to homotopy.
\item If $\varphi$ is a $1$-quasi-isomorphism, then 
$\widehat\varphi$ is an isomorphism.
\end{enumerate}
\end{theorem}

\subsection{$\Delta$-sets, complexes, and cochains} 
\label{subsec:delta} 
A \emph{$\Delta$-set} is a sequence of sets
$X = \{X_n\}_{n \ge 0}$ together with face maps
$d_i \colon X_n \to X_{n-1}$ for $0 \le i \le n$,
satisfying $d_i d_j = d_{j-1} d_i$ for $i < j$.
A \emph{map of $\Delta$-sets} $t \colon X \to Y$ is a collection
of set maps $\{t_n \colon X_n \to Y_n\}$ satisfying
$t_n \circ d_i = d_i \circ t_{n+1}$.
The geometric realization $|X|$ is a CW-complex
\cite{Rourke-Sanderson}, and the assignment $X \leadsto |X|$
is functorial.

The cochain complex $C^*(X; R)$ is the dual of the simplicial
chain complex of $|X|$: the chain group $C_n(X; \Z)$ is free
abelian on $X_n$, with boundary
$\partial_n = \sum_{i=0}^n (-1)^i d_i$, and
$C^n(X; R) = \Hom(C_n(X;\Z), R)$.
The coboundary $d\colon C^n(X;R)\to C^{n+1}(X;R)$ is the dual of
$\partial_{n+1}$: for $\alpha\in C^n(X;R)$ and $\sigma\in X_{n+1}$,
\begin{equation}
\label{eq:coboundary}
(d\alpha)(\sigma) = \sum_{i=0}^{n+1} (-1)^i\, \alpha(d_i \sigma).
\end{equation}
In particular, for $u\in C^1(\Delta^2;R)$ evaluated on the standard
$2$-simplex $(0,1,2)$,
\[
(du)(0,1,2) = u(1,2) - u(0,2) + u(0,1).
\]

Recall from~\cite{Porter-Suciu-2023} that the
\emph{bar construction} $B(G)$ on a group $G$ is the $\Delta$-set
with $B(G)_0 = \{v\}$ and $B(G)_n$ the set of $n$-tuples
$[g_1 | g_2 | \cdots | g_n]$ for $n \ge 1$,
with face operators
\begin{equation}
\label{eq:faceBG}
\begin{split}
d_0([g_1 | \cdots | g_n]) &= [g_2 | \cdots | g_n],\\
d_i([g_1 | \cdots | g_n])
  &= [g_1 | \cdots | g_{i-1} | g_i g_{i+1} | g_{i+2}
     | \cdots | g_n], \quad 1 \le i \le n-1,\\
d_n([g_1 | \cdots | g_n]) &= [g_1 | \cdots | g_{n-1}].
\end{split}
\end{equation}
The faces of the standard $n$-simplex $(0, \ldots, n)$ are
\begin{equation}
\label{eq:faceDelta}
\begin{split}
d_0(0, \ldots, n) &= (1, \ldots, n),\\
d_i(0, \ldots, n) &= (0, \ldots, i-1, i+1, \ldots, n),
  \quad 1 \le i \le n-1,\\
d_n(0, \ldots, n) &= (0, \ldots, n-1).
\end{split}
\end{equation}

The cup product on $C^*(X;R)$ is defined by the
Alexander--Whitney formula: for $\alpha \in C^p(X;R)$, 
$\beta \in C^q(X;R)$, and a $(p+q)$-simplex 
$\sigma = (v_0, \ldots, v_{p+q})$,
\begin{equation}
\label{eq:cup-product}
(\alpha \cup \beta)(\sigma) = 
\alpha(v_0, \ldots, v_p) \cdot \beta(v_p, \ldots, v_{p+q}),
\end{equation}
where the dot denotes multiplication in $R$. In particular, 
for $\alpha, \beta \in C^1(\Delta^2;R)$,
\[
(\alpha \cup \beta)(0,1,2) = \alpha(0,1) \cdot \beta(1,2);
\]
this is the only nontrivial pairing of two $1$-cochains on $\Delta^2$.
The formula~\eqref{eq:cup-product} extends $R$-bilinearly to the full
cochain complex, and together with the coboundary~\eqref{eq:coboundary}
makes $C^*(X;R)$ a differential graded $R$-algebra.

The cochain algebra $C^*(X;R)$ carries additional structure: it is
in fact a $\cup_1$-dga in the sense used throughout the paper.
Steenrod's cup-one product~\cite{Steenrod} is defined, 
for $u \in C^p(X;R)$, $v \in C^q(X;R)$, on the standard 
$(p+q-1)$-simplex $[0, 1, \ldots, p+q-1]$ by
\begin{equation}
\label{eq:cup1-steenrod}
\begin{split}
(u \cup_1 v)([0,\ldots,p+q-1]) &= 
\sum_{j=0}^{p-1} (-1)^{(p-j)(q+1)}\\[-4pt]
& \hspace{1em} u([0,\ldots,j,j+q,\ldots,p+q-1]) \cdot 
v([j,\ldots,j+q]).
\end{split}
\end{equation}
In particular, $u \cup_1 v = 0$ whenever $u$ or $v$ has degree $0$,
and for $u, v \in C^1(X;R)$ and a $1$-simplex $e$,
\[
(u \cup_1 v)(e) = u(e) \cdot v(e).
\]
With these operations, $C^*(X;R)$ is a $\cup_1$-dga 
\cite[Thm.~4.4]{Porter-Suciu-2021}; this is the geometric origin 
of the algebraic structure axiomatized in 
\S\ref{subsec:binomial-cup1}. Moreover, the construction is 
functorial: a map $f \colon X \to Y$ of $\Delta$-complexes induces 
a cochain map $f^\sharp \colon C^*(Y;R) \to C^*(X;R)$ commuting 
with both $\cup$ and $\cup_1$~\cite[Thm.~3.1]{Steenrod}.


\end{document}